\documentclass{amsart}
\usepackage{amssymb}
\usepackage{amscd}
\usepackage{amsmath}
\usepackage{amsfonts}
\usepackage{latexsym}
\usepackage[all]{xy}
\usepackage{mathbbol}
\usepackage{hyperref}
\usepackage{verbatim}

\newtheorem{theorem}{Theorem}[section]
\newtheorem{lemma}[theorem]{Lemma}
\newtheorem{proposition}[theorem]{Proposition}
\newtheorem{corollary}[theorem]{Corollary}

\theoremstyle{definition}
\newtheorem{definition}[theorem]{Definition}
\newtheorem{example}[theorem]{Example}

\newtheorem{question}[theorem]{Question}
\newtheorem{remark}[theorem]{Remark}

\DeclareMathOperator{\ev}{\operatorname{\mathsf{ev}}}
\DeclareMathOperator{\coev}{\operatorname{\mathsf{coev}}}
\newcommand{\Tr}{\operatorname{\mathsf{Tr}}}%{\text{Tr}\,}
\newcommand{\Trace}{\operatorname{\mathsf{Trace}}}
\newcommand{\id}{\operatorname{\mathsf{id}}}

\newcommand{\FPdim}{\operatorname{\mathsf{FPdim}}} 
\newcommand{\Ker}{\operatorname{\mathsf{Ker}}} 
\newcommand{\Gal}{\operatorname{\mathsf{Gal}}}

\newcommand{\Fun}{\operatorname{\mathsf{Fun}}} 
\newcommand{\End}{\operatorname{\mathsf{End}}} 
\renewcommand{\Vec}{\operatorname{\mathsf{Vec}}} 
\newcommand{\sVec}{\operatorname{\mathsf{sVec}}}

\newcommand{\Hom}{\text{Hom}}

\newcommand{\Aut}{\operatorname{\mathsf{Aut}}}

\newcommand{\Rep}{\operatorname{\mathsf{Rep}}}

\newcommand{\rev}{\text{rev}}

\newcommand{\op}{\text{op}}

\newcommand{\eps}{\varepsilon}

\newcommand{\B}{\mathcal{B}}
\newcommand{\C}{\mathcal{C}}
\newcommand{\K}{\mathcal{K}}
\newcommand{\D}{\mathcal{D}}

\newcommand{\E}{\mathcal{E}}

\newcommand{\Z}{\mathcal{Z}}

\renewcommand{\L}{\mathcal{L}}
\newcommand{\M}{\mathcal{M}}
\newcommand{\A}{\mathcal{A}}
\newcommand{\N}{\mathcal{N}}

\renewcommand{\O}{\mathcal{O}}

\newcommand{\be}{\mathbf{1}}

\newcommand{\g}{\mathfrak{g}}

\newcommand{\la}{\langle\,}
\newcommand{\ra}{\,\rangle}

\renewcommand{\be}{\mathbf{1}}

\newcommand{\bt}{\boxtimes}
\newcommand{\ot}{\otimes}

\newcommand{\kk}{\mathbb{k}}

\begin{document}
\title[Morita equivalence methods]{Morita equivalence methods in classification of  fusion categories}

\author{Dmitri Nikshych}
\address{Department of Mathematics and Statistics,
University of New Hampshire,  Durham, NH 03824, USA}
\email{nikshych@math.unh.edu}

\begin{abstract}
We describe an  approach to classification of fusion categories  in terms of Morita
equivalence.  This is usually achieved by analyzing Drinfeld centers of fusion categories
and finding Tannakian subcategories therein.
\end{abstract}
\maketitle

\setcounter{tocdepth}{1}
\tableofcontents
\nopagebreak

%%%%%%%%%%%%%%%%%%%%%%%%%%%%%%%%%%%%%%%%%%%%%%%
%%%%%%%%%%%%%%%%%%%%%%%%%%%%%%%%%%%%%%%%%%%%%%%
%%%%%%%%%%%%%%%%%%%%%%%%%%%%%%%%%%%%%%%%%%%%%%%
\section{Introduction}

The purpose of this survey is to describe an approach to classification of fusion categories
using categorical duality (i.e., categorical Morita equivalence).  This duality is a categorical analogue
of the following classical construction in algebra: given a ring $R$ and a left  $R$-module $M$
one has the ring $R^*_M:= \End_R(M)$ of $R$-linear endomorphisms of $M$. We can view this
ring as the dual ring of $R$ with respect to $M$.  Two rings $R$ and $S$ are called {\em 
Morita equivalent} if $S^\op$ is isomorphic to $R^*_M$ with respect to some progenerator 
module $M$
(here  $S^\op$ denotes the ring with the opposite multiplication).  It is well known
that Morita equivalent rings have equivalent Abelian categories of modules.

In the categorical setting one replaces rings by tensor categories, their modules
by module categories,  and module endomorphisms by module endofunctors,
see Section~\ref{sect module} for definitions.  Thus, given, a tensor category $\A$
and  a left $\A$-module category $\M$ one has a new tensor category $\A^*_\M$
which we call   {\em categorically Morita equivalent} or {\em dual} to $\A$ with respect to $\M$.  
If $\A$ is a fusion category and $\M$ is an indecomposable semisimple $\A$-module category 
then $\A^*_\M$ is also a fusion category. 

One can produce new examples of fusion categories in this way.  Namely, starting with
a known fusion category $\A$ one finds its module categories and construct duals. 
For example, when $\A$ is a pointed (respectively, nilpotent)  category  (see Definitions~\ref{def:pointed}
and \ref{def:nilpotent}) the resulting dual categories  form an important class of fusion categories 
called {\em group-theoretical} 
(respectively,  {\em weakly group-theoretical} )   categories.

In the opposite direction, given a class of fusion categories (e.g., of a given dimension)
one can try  to show that categories in this class are categorically Morita equivalent 
to some well understood categories.  In other words, one tries to classify 
fusion categories up to a categorical Morita equivalence.  
Classification results of this type for fusion categories of small integral dimension
were obtained in \cite{DGNO1, ENO2, Nat}. The techniques
used in these papers are based on recovering some group-theoretical information about the
Morita equivalnce class of a given fusion category. 

In this paper we try to explain the ideas and methods used in the above  classification.
Of particular importance is the structural theory of braided categories. 
This comes from the fact that the braided equivalence class of the 
Drinfeld center of a fusion category is its complete Morita equivalence invariant, 
see Theorem~\ref{2007 criterion}.

This paper does not contain new results or proofs. Its only goal is to collect relevant notions
and facts in one place.  In Sections 2, 3, and 4 we recall definitions and basic facts
about fusion categories, their module categories, and braided categories.  In Section 5
we describe  a connection between the structure of the center of a fusion category $\A$
and the Morita equivalence class of $\A$.  These results are applied  
to classification of fusion categories of low dimension in Section 6. 

The author would like to thank  Nicolas Andruskiewitsch, Juan Cuadra, and Blas Torrecillas
for organizing the conference ``Hopf algebras and tensor categories" in beautiful Almeria
in July 2011 and for inviting him to write this survey. The author is also grateful to the referee
for useful comments. 

The author's work  was partially supported by 
the NSF grant DMS-0800545. 

%%%%%%%%%%%%%%%%%%%%%%%%%%%%%%%%%%%%%%%%%%%%%%%
%%%%%%%%%%%%%%%%%%%%%%%%%%%%%%%%%%%%%%%%%%%%%%%
%%%%%%%%%%%%%%%%%%%%%%%%%%%%%%%%%%%%%%%%%%%%%%%
\section{Definitions and basic notions}

Throughout this paper $\kk$ denotes an algebraically closed field of characteristic zero.
All categories  are assumed to be Abelian, semisimple, $\kk$-linear and  have finitely many 
isomorphism classes of simple objects and finite dimensional spaces of morphisms. 
All functors are assumed to be additive and $\kk$-linear.

%%%%%%%%%%%%%%%%%%%%%%%%%%%%%%%%%%%%%%%%%%%%%%%
\subsection{Definitions, basic properties, and examples of fusion categories}

The following definition was given in \cite{ENO1}.

\begin{definition}
\label{fusion category}
A {\em fusion category} over $\kk$ is a rigid tensor category  such that the unit object $\be$ is simple. 
\end{definition}

That is, a fusion category $\A$ is  a category equipped with  tensor product
bifunctor  $\otimes:  \A \times \A \to \A$, the natural isomorphisms (associativity and unit constraints)
\begin{gather}
a_{X,Y,Z}:(X\ot Y)\ot Z \xrightarrow{\sim} X \ot (Y \ot Z), \\
l_X: \be \ot X \xrightarrow{\sim} X, \quad \text{and}\quad  r_X: X \ot \be \xrightarrow{\sim} X, 
\end{gather}
satisfying the following coherence axioms:
\begin{enumerate}
\item[]{\bf 1.\  The Pentagon Axiom.} The diagram
\begin{equation}
\label{pentagon}
\xymatrix{
&((W\ot X)\ot Y)\ot Z  \ar[dl]_{a_{W\ot X, Y, Z}} \ar[dr]^{a_{W,X,Y}\otimes \id_Z}&\\
(W \ot (X \ot Y))\ot Z
\ar[d]_{a_{W, X\ot Y, Z}}   && (W \ot X) \ot (Y \ot Z) \ar[d]^{a_{W, X, Y\ot Z}} \\
 W \ot ((X\ot Y) \ot Z) 
\ar[rr]^{\id_W \ot a_{X,Y,Z}}  && W \ot (X \ot (Y \ot Z))
} 
\end{equation}
is commutative for all objects $W,\,X,\,Y,\,Z$ in $\A$.

\item[]{\bf 2.\ The triangle axiom.} The diagram
\begin{equation}
\label{triangle}
\xymatrix{(X\otimes \be )\otimes Y\ar[rr]^{a_{X,\be ,Y}} \ar[dr]_{r_X\otimes 
\id_Y}&&X\otimes (\be \otimes Y)
\ar[dl]^{\id_X\otimes l_Y}\\ &X\otimes Y&}
\end{equation}
is commutative for all objects $X,\,Y$ in $\A$. 
\end{enumerate}

The coherence theorem of MacLane states that {\em every} diagram
constructed from the associativity and unit isomorphisms commutes.
We will sometimes omit associativity constraints from formulas.

The rigidity condition means that for every object $X$ of $\A$ there exist left and right duals of $X$.
Here a {\em left dual} of $X$ is an object $X^*$ in $\A$ for which 
there exist morphisms $\ev_X:X^*\ot X\to \be$ and
$\coev_X:\be\to X\ot X^*$, called the {\em evaluation} and
{\em coevaluation} such that the compositions
\begin{gather}
\label{left dual axioms}
X \xrightarrow{\coev_X\ot \id_X} (X\ot X^*)\ot X 
\xrightarrow{a_{X,X^*,X}}
X\ot (X^*\ot X) \xrightarrow{\id_X\ot \ev_X} X, \\
\label{left dual axioms'}
X^* \xrightarrow{\id_{X^*}\ot \coev_X} X^*\ot (X\ot X^*) \xrightarrow{a^{-1}_{X,X^*,X}} 
 (X^*\ot X)\ot X^* \xrightarrow{\ev_X \ot \id_{X^*}} X^* 
\end{gather}
are the identity morphisms. A right dual  ${}^*X$ is defined in a similar way.  Dual objects are unique up
to isomorphism. One has  $({}^*X)^* \cong X \cong {}^*(X^*)$ for all objects $X$ in $\C$.
Also,  there exist a (non-canoical) isomorphism $X^* \cong {}^*X$ for every $X$ (indeed,  for a simple $X$ both  $X^*$ and ${}^*X$
are isomorphic to the unique simple object $Y$ such that $\be$ is contained in $X\ot Y$). 

For a fusion category $\A$ let $\A^\op$ denote the fusion category with the opposite 
tensor product. 

\begin{definition}
Let $\A_1,\, \A_2$ be fusion categories. A {\em tensor functor} between $\A_1$ and $\A_2$
is a functor $F: \A_1\to A_2$ along with  natural isomorphisms 
\[
J_{X,Y}: F(X)\ot F(Y)
\xrightarrow{\sim} F(X\ot Y)\quad \mbox{and}  \quad\varphi: F(\be) \xrightarrow{\sim} \be
\] 
such that the diagrams 
\begin{equation}\label{tenfun1}
\xymatrix{
(F(X)\ot F(Y))\ot F(Z) \ar[rrr]^{a_{F(X),F(Y),F(Z)}}
\ar[d]_{J_{X,Y}\ot \id_{F(Z)}} & &  &
F(X)\ot (F(Y)\ot F(Z)) \ar[d]^{\id_{F(X)}\ot J_{Y,Z}} \\
F(X\ot Y)\ot F(Z)    \ar[d]_{J_{X\ot Y,Z} } & & & F(X)\ot F(Y\ot Z)
\ar[d]^{J_{X,Y\ot Z}} \\
F((X\ot Y)\ot Z) \ar[rrr]^{F(a_{X,Y,Z})} & & & F(X\ot (Y\ot Z)),
}
\end{equation}
\begin{equation}
\xymatrix{
F(\be) \ot  F(X) \ar[rr]^{J_{\be, X}}  \ar[d]_{\varphi\ot \id_{F(X)}} && F(\be \ot X) \ar[d]^{F(l_X)} \\
\be \ot F(X) \ar[rr]_{l_{F(X)}}  && F(X),
}
\end{equation}
and
\begin{equation}
\xymatrix{
F(X) \ot  F(\be) \ar[rr]^{J_{X,\be}}  \ar[d]_{\id_{F(X)} \ot \varphi } && F(X \ot \be) \ar[d]^{F(r_X)} \\
F(X)\ot \be  \ar[rr]_{r_{F(X)}}  && F(X).
}
\end{equation}
commute for all objects $X,\,Y,\, Z$ in $\A_1$. 
\end{definition}

%\begin{remark}
%The actual diagrams for the pentagon and triangle axioms and the compatibility condition
%for tensor functors can be found in many textbooks, e.g., in \cite[Section 1.1]{BK},
%\cite[Section XI.2]{Ka},  and \cite{Mac}. 
%\end{remark}

\begin{definition}
Let $\A$ and $\B$ be fusion categories and let
 $F^1, \, F^2 : \A \to \B$ be tensor functors between fusion categories
with tensor structures 
\[
J^i_{X,Y}: F^i(X)\ot F^i(Y) \xrightarrow{\sim} F^i(X\ot Y),\quad  i=1,2.
\]
A  natural  morphism $\eta$ between $F^1$ and $F^2$ is called {\em tensor} if  its components 
\[
\eta_X: F^1(X) \to F^2(X)
\]
 satisfy  the commutative diagram
\begin{equation}\label{tenfunmor1}
\xymatrix{
F^1(X) \ot F^1(Y)  \ar[rr]^{ \eta_X \ot \eta_Y} \ar[d]_{J^1_{X,Y}} && F^2(X)\ot F^2(Y)
\ar[d]^{J^2_{X,Y}} \\
F^1(X\ot Y) \ar[rr]^{\eta_{X\ot Y}} && F^2(X\ot Y),
}
\end{equation}
for all $X,\,Y\in \A$.
\end{definition}

Tensor autoequivalences of a fusion category $\A$  form a monoidal category denoted $\Aut_\ot(\A)$.

Let $G$ be a finite group and let ${\underline G}$ denote the
monoidal category whose objects are elements of $G$, morphisms are
identities, and the tensor product is given by the multiplication in
$G$. 

\begin{definition}
\label{def:action}
An {\em action} of $G$ on  a fusion category $\A$ is a
monoidal functor 
\begin{equation}
\label{Tg}
T: {\underline G} \to \Aut_\ot(\A): g \mapsto T_g.  
\end{equation}
\end{definition}

This means that for every $g\in G$ there is a tensor autoequivalence $T_g:\A\to \A$ 
and for any pair $g,\, h\in G$, there is a  natural isomorphism of tensor functors 
\[
\gamma_{g,h}: T_g\circ T_h\simeq T_{gh}
\]  
satisfying usual compatibility conditions.

Note that for any fusion category $\A$ the functor
\[
\A \to \A^\op : X \mapsto X^*
\]
is a tensor equivalence. Consequently,  the functor
\[
\A \to \A : X \mapsto X^{**}
\]
is a tensor autoequivalence of $\A$.

\begin{definition}
\label{def:pivotal}
A {\em pivotal structure} on a  fusion category $\A$ is a tensor isomorphism $\psi$ between the identity
autoequivalence of $\A$ and the functor $X \mapsto X^{**}$ of taking the second dual.  A  fusion
category with a pivotal structure is called {\em pivotal}. 
\end{definition}

In a pivotal category there is a notion of a {\em trace}
of an endomorphism. Namely, for   $f\in \End_\A(X)$ set:
\begin{equation}
\label{old friend}
\Tr(f) : \be \xrightarrow{\coev_X} X \ot X^* \xrightarrow{\psi_X\circ f \ot \id_{X^*}} X^{**}\ot X^* \xrightarrow{\ev_{X^*}} \be,  
\end{equation}
so that $\Tr(f)\in \End_\A(\be)=\kk$. The {\em dimension} of $X\in \A$ is defined by 
\begin{equation}
\label{eqn:dim}
d_X =\Tr(\id_X). 
\end{equation}
Note that $d_X\neq 0$ for every simple $X$.
A pivotal structure (respectively, a pivotal category) is called {\em spherical} if $d_X= d_{X^*}$
for all objects $X$. 

\begin{remark}
It is not known whether  every fusion category has a pivotal (or spherical) structure.  
It is true for pseudo-unitary categories, see  Proposition~\ref{canonical spherical structure}.
In particular  it is true for categories of integer Frobenius-Perron dimension 
(Corollary~\ref{canonical spherical structure in integer case}). 
%However, it is known 
%\cite{ENO4} that for  any fusion category there is canonical tensor isomorphism between the 
%identity autoequivalence and the functor $X \mapsto X^{****}$ of taking the  fourth dual. 
\end{remark}

We say that  a tensor functor $F: \A_1\to A_2$  is {\em injective} if it is fully faithful and {\em surjective}
if for any object $Y\in \A_2$ there is an object $X\in \A_1$ such that $Y$ is isomorphic to a direct
summand of $F(X)$. In the latter case we call $\A_2$ a quotient category of $\A$.

By a {\em fusion subcategory} of a fusion category $\A$
we always mean a full tensor subcategory $\tilde{\A} \subset \A$ 
such that if $X\in \tilde{\A}$ is isomorphic to a direct summand of an object of $\tilde{\A}$
then $X\in \tilde{\A}$. As an additive category,  $\tilde{\A}$ is generated by some
of the simple objects of $\A$.  It is known, see \cite[Appendix F]{DGNO2}, that a fusion subcategory
of a fusion category is rigid; therefore, it is itself a fusion category.

Let $\Vec$ denote the fusion category of finite dimensional vector spaces over $\kk$.
A tensor functor $\A \to \Vec$ will be called a {\em fiber functor}.

Any fusion category $\A$ contains a trivial fusion subcategory consisting of multiples
of the unit object $\be$. We will  identify  this subcategory with $\Vec$. 
%A fusion category $\A$
%is called {\em simple} if $\Vec$ is the only proper fusion subcategory of $\A$.

\begin{definition}
\label{def:pointed}
A fusion category is called {\em pointed} if all its simple objects are invertible
with respect to tensor product.
For a fusion category $\A$ we denote  $\A_{pt}$ the maximal pointed fusion subcategory
of $\A$. %We say that $\A$ is {\em unpointed} if $\A_{pt} =\Vec$.
\end{definition}

We will denote $\A \boxtimes \B$ the {\em tensor product} of fusion categories $\A$ and $\B$ 
\cite[Section 5]{De}.  
The category $\A \boxtimes \B$ is  obtained 
as the completion of the $\kk$-linear direct product $\A\otimes_\kk\B$ under direct sums and subobjects.

%%%%%%%%%%%%%%%%%%%%%%%%%%%%%%%%%%%%%%%%%%%%%%%
\subsection{First examples of fusion categories}
\label{sect first examples}

Let $G$ be a finite group.

\begin{example} 
\label{ex VecGw}
The following is the most general example of a pointed fusion category.
Let $\omega$ be a normalized $3$-cocycle on $G$ with values in $\kk^\times$,  the multiplicative
group of the ground field. That is,
$\omega: G\times G\times G\to \kk^\times$ is a function satisfying equations
\begin{equation}
\label{a 3-cocycle}
\omega(g_1g_2, g_3, g_4)\omega(g_1, g_2, g_3 g_4) =
\omega(g_1, g_2, g_3) \omega(g_1, g_2 g_3, g_4) \omega(g_2, g_3, g_4)
\end{equation}
and 
\begin{equation}
\omega(g_1,\, 1,\, g_2) =1,
\end{equation}
for all $g_1, g_2, g_3, g_4\in G$. 

Let $\Vec_G^\omega$ denote the category of $G$-graded $\kk$-vector 
spaces with the tensor product of objects $U=\oplus_{g\in G}\, U_g$ and $V=\oplus_{g\in G}\, V_g$
given by
\[
(U\ot V)_g =\bigoplus_{xy =g}\, U_x\ot V_y,\qquad g \in G,
\] 
with the associativity constraint $a_{U,V,W}: (U\ot V) \ot W \xrightarrow{\sim} U \ot (V \ot W)$
on homogeneous spaces $U,\,V,\, W$ of degrees $g_1,\,g_2,\,g_3\in G$ given by $\omega(g_1,\,g_2,\,g_3)$
times the canonical vector spaces isomorphism $(U\ot V) \ot W \xrightarrow{\sim} U \ot (V \ot W)$.

For any $g\in G$ let $\delta_g$ denote the corresponding simple object of $\Vec_G^\omega$.
We have
\[
\delta_g \ot \delta_h =\delta_{gh},\qquad g,h\in G.
\]

Two categories $\Vec_G^\omega$ and  $\Vec_{\tilde{G}}^{\tilde\omega}$ are equivalent if  and only if 
there is a group isomorphism $f: G \to \tilde{G}$ such that $\omega$ and $\tilde\omega\circ (f\times f \times f)$
are cohomologous $3$-cocycles on $G$. 

When $\omega=1$ we will denote the corresponding pointed fusion category by $\Vec_G$. 
\end{example}

\begin{example}
\label{ex RepG}
Let $\Rep(G)$ be the category of finite dimensional representations of $G$ over $\kk$.  It is a fusion category 
with the usual tensor product and associativity isomorphisms.  The unit object is the trivial
representation.  The left and right dual objects of the representation $V$ are both given by the dual representation $V^*$. 
Simple objects of $\Rep(G)$ are irreducible representations.

The category $\Rep(G)$  is pointed if and only if $G$ is Abelian, in which case there is
a canonical equivalence $\Rep(G) \cong \Vec_{\widehat{G}}$, where $\widehat{G}$ is the 
group of characters of~$G$.
\end{example}

\begin{example}
\label{ex Hopf}
This is a generalization of Example~\ref{ex RepG}.
Let $H$ be a  semisimple Hopf algebra over $\kk$ (such an algebra is automatically finite dimensional)
with the comultiplcation $\Delta: H \to H\ot H$, antipode $S: H \to H$, and counit $\eps : H \to \kk$ \cite{Mo}. 
The category $\Rep(H)$  of finite dimensional $H$-modules is a fusion category with the tensor
product of $H$-modules $V,\, W$ being $V\ot_\kk W$ with the action of $H$ given by
\[
h \cdot (v \ot w) = \Delta(h)(v \ot w),\qquad  v\in V,\, w\in W,\,h\in H.
\]
The unit object is $\kk$ with the action given by $\eps$.  The left dual of an $H$-module
$V$  is the  dual vector space $V^*$ with an $H$-module structure given by 
$\langle h.\phi,\, v \rangle  = \langle \phi,\, S(h).v \rangle$.  

Let $H_1,\, H_2$ be semisimple Hopf algebras.  A   homomorphism $f:~H_1\to H_2$ of Hopf
algebras induces a tensor functor $F: \Rep(H_2) \to \Rep(H_1)$.
The functor $F$ is injective (respectively, surjective) if and only if $f$ is surjective (respectively,
injective).
\end{example}

\begin{remark}
Note that there is forgetful tensor functor $\Rep(H)\to \Vec$, i.e., a {\em fiber functor}.
Conversely, if $\A$ is a fusion category that admits a fiber functor then $\A \cong \Rep(H)$
for some semisimple Hopf algebra $H$ \cite{Ul}. 
\end{remark}

\begin{example}
\label{ex qHopf}
Similarly, if $Q$ is a semisimple quasi-Hopf algebra then $\Rep(Q)$ is a fusion category.
In this case the assocativity constraint is determined by the associatior $\Phi\in Q\ot Q \ot Q$.
The category $\Vec_G^\omega$ from Example~\ref{ex VecGw} is a special case of such category.
\end{example}

\begin{example}
Let $\g$ be a finite dimensional simple Lie algebra and let $\hat \g$ be the
corresponding affine Lie algebra. For any $k\in \mathbb{Z}_{>0}$ let $\C(\g,k)$  the
category of highest weight integrable $\hat \g-$modules of level $k$ is a fusion category, 
see e.g. \cite[Section~ 7.1]{BK}  where this category is denoted $\O_k^{int}$. 
\end{example}

%%%%%%%%%%%%%%%%%%%%%%%%%%%%%%%%%%%%%%%%%%%%%%%
\subsection{Group theoretical constructions: extensions and equivariantizations}
\label{GT construction sect}

For a fusion category $\A$ let $\O(\A)$ denote the set of (representatives of isomorphism classes of)
simple objects of $\A$.

\begin{definition}
\label{def:grading}
A {\em grading} of $\A$ by a group $G$ is a map $\deg: \O(\A) \to G$ with the following property:
for all simple objects $X,\,Y,\,Z\in \A$ such that $X\ot Y$ contains $Z$ one has $\deg Z = \deg X \cdot \deg Y$.
\end{definition}

The name ``grading" is also used for the corresponding decomposition 
\begin{equation}
\label{G-graded category}
\A = \bigoplus_{g\in G}\, \A_g,
\end{equation}
where $\A_g$ is the  full additive subcategory generated by simple  objects of degree $g\in G$. 
We say that a grading is {\em faithful}   if $\A_g\neq 0$ for all $g\in G$.
Note that the {\em trivial component} $\A_e$  of  the grading \eqref{G-graded category}
is a fusion subcategory of $\A$. 

\begin{definition}
\label{def:adjoint}
Let $\A$ be a fusion category. The {\em adjoint} subcategory $\A_{ad} \subset \A$
is the fusion subcategory of $\A$ generated by  objects $X\ot X^*,\, X\in \O(\A)$.
\end{definition}

It was explained in \cite{GN} that there exists a group $U(\A)$, called the {\em universal grading group} of $\A$,
and a faithful grading 
\[
\A =\bigoplus _{g\in U(\A)}\, \A_g \qquad \mbox{with } \A_e=\A_{ad}.
\]
This grading is universal in the sense that any  faithful 
grading \eqref{G-graded category} of $\A$ is obtained by taking a quotient of the group $U(\A)$.

\begin{example}
Let $\A = \Rep(H)$, where $H$ is a semisimple Hopf algebra. 
Let $K$ be the maximal Hopf subalgebra of $H$ contained in the center of $H$. 
Then $K$ is isomorphic to the Hopf algebra of functions on $U(\Rep(H))$. In other words,
the universal grading group of 
$\Rep(H)$ is the spectrum of $K$ \cite{GN}.
\end{example}

\begin{definition}
\label{def:extension}
Let $G$ be a finite group. We say that a fusion category $\A$
is a {\em $G$-extension} of a fusion category $\B$ if there is a faithful $G$-grading of $\A$ such that 
$\A_e \cong \B$.  
\end{definition}

\begin{remark}
A classification of $G$-extensions of fusion categories is obtained in \cite{ENO3}.  Namely, $G$-extensions
of a fusion category $\A$  correspond to homomorphisms $G \to \text{BrPic}(\A)$, where $\text{BrPic}(\A)$
is the {\em Brauer-Picard} group of $\A$ consisting of invertible $\A$-module categories, and certain cohomological
data. See \cite{ENO3} for details.  
This theory extends the clasical theory of strongly graded rings (also known as generalized crossed products) 
and their description using bimodules and crossed systems.
\end{remark}

\begin{definition}
\label{def:nilpotent}
A fusion category $\A$ is called {\em nilpotent}  \cite{GN} if there is a sequence of finite groups
$G_1,\dots, G_n$ and a sequence of fusion subcategories of $\A$:
\[
\A_0 =\Vec \subset \A_1\subset \cdots \subset \A_n =\A,
\]
such that $\A_{i}$ is a $G_i$-extension of  $\A_{i-1},\, i=1,\dots,n$.  The smallest such $n$
is called the {\em nilpotency class} of  $\A$.  A nilpotent fusion category is called {\em cyclically nilpotent}
if all groups $G_i$ are cyclic.
\end{definition} 

\begin{remark}
The category $\Vec_G^\omega$ is nilpotent for any finite group $G$ (since pointed fusion categories are precisely 
nilpotent fusion categories of class $1$). The category $\Rep(G)$ is nilpotent if and only if $G$ is nilpotent.
\end{remark}

Let $G$ be a group acting on a fusion category $\A$, see Definition~\ref{def:action}. 

\begin{definition}
\label{Gequiv object} A {\em $G$-equivariant object\,} in $\A$ is a
pair $(X,\{u_g\}_{g\in G})$ consisting of an object $X$ of $\A$
together with a collection of isomorphisms
 $u_g: T_g(X)\simeq X,\, g\in G$, such that the diagram
\begin{equation*}
\label{equivariantX}
\xymatrix{T_g(T_h(X))\ar[rr]^{T_g(u_h)} \ar[d]_{\gamma_{g,h}(X)
}&&T_g(X)\ar[d]^{u_g}\\ T_{gh}(X)\ar[rr]^{u_{gh}}&&X}
\end{equation*}
commutes for all $g,h\in G$.
One defines morphisms of equivariant objects to be morphisms in $\A$ commuting with $u_g,\; g\in G$.
\end{definition}

Equivariant objects in $\A$ form a fusion category, called the  {\em equivariantization}
of $\A$ and denoted by $\A^G$.  There is a natural forgetful tensor functor
$\A^G \to \A$.

\begin{example}
\label{equiv examples}
Let $G$ be a finite group.
\begin{enumerate}
\item[(i)] Consider $\Vec$ with the trivial action of $G$. Then $\Vec^G \cong \Rep(G)$.
\item[(ii)] More generally, let $N$ be a normal subgroup of $G$.  The corresponding action 
of $G/N$ on $N$ induces an action of $G/N$ on $\Rep(N)$. We have $\Rep(N)^{G/N} \cong \Rep(G)$.
\item[(iii)] Consider $\Vec_G$ with the action of $G$ by conjugation. Then $(\Vec_G)^G \cong \Z(\Vec_G)$
is the center of $\Vec_G$, cf.\  Example~\ref{ZVecG}. 
\end{enumerate}
\end{example}

%%%%%%%%%%%%%%%%%%%%%%%%%%%%%%%%%%%%%%%%%%%%%%%
\subsection{The Grothendieck ring and Frobenius-Perron dimensions}

As before, let $\O(\A)$ denote the set of
isomorphism classes of simple objects in a fusion category~$\A$.

For any object $X$ of $\A$  and any $Y\in \O(\A)$ let $[X:Y]$ denote the multiplicty of $Y$ in $X$.

The {\em Grothendieck ring} $K(\A)$ of a fusion category $\A$ is
generated by isomorphism classes of objects  $X\in \A$ with the addition and multiplication
given by 
\[
X+Y = X\oplus Y \quad \mbox{and} \quad XY = X\otimes Y
\] 
for all $X,\,Y\in \A$.  Clearly, $K(\A)$ is a free $\mathbb{Z}$-module with basis $\O(\A)$.

\begin{example}
For the pointed fusion category $\Vec_G^\omega$ from Example~\ref{ex VecGw}  we have
$K(\Vec_G^\omega) =  \mathbb{Z}G$. 
\end{example}

There exists a unique ring
homomorphism $\FPdim: K(\A)\to \mathbb{R}$, called the {\em Frobenius-Perron dimension}  
such that $\FPdim(X)>0$ for any $0\ne X\in \A$, see
\cite[Section 8.1]{ENO1}. The number $\FPdim(X)$ is  the largest positive eigenvalue
(i.e., the Frobenius-Perron eigenvalue) of the integer non-negative matrix 
\[
N^X = (N^Z_{XY})_{Y,Z\in \O(\A)},
\]
where
\[
X \ot Y \cong \bigoplus_{Z\in \O(\A)}\, N^Z_{XY} \, Z.
\]
For every object $X\in \A$ we have
\[
\FPdim(X) =\FPdim(X^*).
\]

For a fusion category $\A$ one defines (see \cite[Section  8.2]{ENO1}) its
{\em Frobenius-Perron dimension}:
\begin{equation}
\label{FPdim def}
\FPdim(\A)=\sum_{X\in \O(\A)}\,\FPdim(X)^2.
\end{equation}

We define the {\em virtual regular object} of $\A$ as 
\[
R_\A := \sum_{X\in \O(\A)}\, \FPdim(X) X \in K(\A)\ot_\mathbb{Z}\mathbb{R}.
\]
%We have $\FPdim(\A) = \FPdim(R_A)$, where we extend $\FPdim$ from $K(\A)$
%to $K(\A)\ot_\mathbb{Z}~\mathbb{R}$.  
Note that $R_\A$ is the unique (up to a
non-zero scalar multiple) element of $K(\A)\ot_\mathbb{Z} \mathbb{R}$
such that $X R_\A = R_A X = \FPdim(X) R_\A$ for all $X \in K(\A)$.

Clearly, $\FPdim(\A)$ and  $\FPdim(X)$ for non-zero $X\in \A$  are positive algebraic integers.
It was shown in \cite[Corollary 8.54]{ENO1} that they are, in fact, cyclotomic integers. 

\begin{definition}
\label{def:integral}
We say that a fusion category $\A$ is {\em integral} if $\FPdim(X)\in \mathbb{Z}$ for every
object $X \in \A$.  
\end{definition}

\begin{remark}
The Frobenius-Perron dimensions in categories considered in Examples~\ref{ex VecGw} -- \ref{ex qHopf}
coincide with vector space dimensions, so these categories are integral. 
Conversely, if $\A$ is an integral  fusion category 
then $\A$ is equivalent to the representation category of some semisimple quasi-Hopf algebra $Q$
(note that this $Q$ is not unique) by \cite[Theorem 8.33]{ENO1}. In this case we can view $R_\A$
as an element of $K(\A)$, namely as the class of the regular representation of $Q$. 
\end{remark}

\begin{remark}
Note a difference between Frobenius-Perron dimensions and dimensions defined by formula
\eqref{old friend}. The former takes values in $\mathbb{R}$ while the latter take values in $\kk$.
So these dimensions are not equal in  general.
\end{remark}

\begin{theorem}
\label{thm Lagrange functor}
Let $F :\A \to \B$ be a surjective tensor functor between fusion categories. Then
the ratio $\FPdim(\A)/\FPdim(\B)$  is an algebraic integer $\geq 1$.
\end{theorem}
\begin{proof}
Note that $F$ induces  a surjective algebra homomorphism  
\[
f: K(\A)\ot_\mathbb{Z} \mathbb{R}\to K(\B)\ot_\mathbb{Z} \mathbb{R}. 
\]
We have $f(R_\A) = a R_\B$ for some  non-zero $a\in \mathbb{R}$.  Computing the multiplicty of $\be$
in both sides of the last equality we have
\[
a = \sum_{X\in \O(\A)}\, \FPdim(X) [F(X):1]. 
\]
On the other hand, since $f$ preserves  Frobenius-Perron dimensions, we have 
\[
a= \frac{\FPdim(R_\A)}{\FPdim(R_\B)} = \frac{\FPdim(\A)}{\FPdim(\B)}.  
\]
Comparing last two equalities gives the result.
\end{proof}

\begin{remark}
In  Theorem~\ref{thm Lagrange functor} one has $\FPdim(\A) = \FPdim(\B)$
if and only if $F$ is an equivalence. 
\end{remark}

A consequence of Theorem~\ref{thm Lagrange functor} is the following formula for the Frobenius-Perron
dimension of the equivariantization category:
\[
\FPdim(\A^G) = |G|\, \FPdim(\A). 
\]

The following result is an analogue  of Lagrange's Theorem in theory of groups and Hopf algebras,
see \cite[Proposition 8.15]{ENO1} and  \cite[Theorem 3.47]{EO}.

\begin{theorem}
\label{thm Lagrange}
Let $\A$ be a fusion category and let $\B \subset \A$ be a fusion subcategory. Then 
the ratio $\FPdim(\B)/\FPdim(\A)$  is an algebraic integer $\leq 1$.
\end{theorem}

In particular, if $\A$ is a faithful $G$-extension of $\B$ then 
\[
\FPdim(\A) = |G|\, \FPdim(\B). 
\]
%
%\begin{proof}
%\end{proof}

%%%%%%%%%%%%%%%%%%%%%%%%%%%%%%%%%%%%%%%%%%%%%%%
\subsection{Ocneanu's rigidity}

The statement that a fusion category cannot be deformed is known as the Ocneanu
rigidity because its formulation and proof for unitary categories was suggested (but not published)
by Ocneanu. The following result was proved in \cite[Theorems 2.28 and 2.31]{ENO1}.

\begin{theorem}
\label{thm rigidity}
The number of fusion categories with a given Grothendieck ring is finite. 
The number of (equivalence classes of) tensor functors between a fixed pair
of fusion categories is finite.
\end{theorem}

\begin{remark}
It follows from  Theorem~\ref{thm rigidity} that every fusion category $\A$ is defined over an algebraic number field.
That is,  the structure constants ($6$j symbols) of $\A$ can be written using algebraic numbers.  Therefore, for any 
$g\in Gal(\bar{\mathbb{Q}}/\mathbb{Q})$  there is a category $g(\A)$, the {\em Galois cojugate} of $\A$ 
obtained from $\A$ by conjugating its structure constants ($6$j symbols) by $g$. 
\end{remark}

\begin{corollary}
For every  positive number $M$  the number of fusion categories whose Frobenius-Perron dimension
is $\leq M$ is finite. 
\end{corollary}
\begin{proof}
In view of Theorem~\ref{thm rigidity} it suffices to show that the number of possible Grothendieck
rings of fusion categories of  Frobenius-Perron dimension $\leq M$ is finite. 
For any simple objects $X,\,Y,\, Z$
let $N_{XY}^Z$ denote the multiplicity of $Z$ in $X\ot Y$.  In any fusion category
of Frobenius-Perron dimension $\leq M$ we have
\[
N_{XY}^Z \leq \frac{\FPdim(X)\FPdim(Y)}{\FPdim(Z)}  \leq \frac{M\,\FPdim(Y)}{\FPdim(Z)},
\]
whence  $(N_{XY}^Z)^2 =   N_{XY}^Z N_{X^*Z}^Y \leq  M^2$.  Thus the structure constants 
of the Grothen\-dieck rings of the above class of fusion categories are uniformly bounded by $M$, 
so there are  finitely many Grothendieck rings.
\end{proof}

%%%%%%%%%%%%%%%%%%%%%%%%%%%%%%%%%%%%%%%%%%%%%%%
\subsection{Categorical dimension and pseudo-unitary categories}

Let $\A$ be a fusion category and let $X$ be an object of $\A$.
Given a morphism $a_X: X \to X^{**}$ we define its trace $\Tr(a_X)$
similarly to how it was done in  \eqref{old friend}:
\[
\Tr(a_X) : \be \xrightarrow{\coev_X} X \ot X^* \xrightarrow{a_X \ot \id_{X^*}} X^{**}\ot X^* \xrightarrow{\ev_X} \be.
\]
If $X$ is simple and $a_X: X \to X^{**}$ is an isomorphism then $\Tr(a_X)\neq 0$ 
and the quantity $|X|^2 =  \Tr(a_X) \Tr((a_X^{-1})^*) \neq 0$ does not depend on the choice of $a_X$
(here $|X|^2$ is regarded as  a single symbol and not as the square of a modulus).

Define the {\em categorical dimension} of $\A$ by
\begin{equation}
\label{dim def}
\dim(\A) = \sum_{X \in \O(\A)}\, |X|^2.
\end{equation}
Suppose $\kk = \mathbb{C}$, the field of complex numbers.  Then one can show that $|X|^2 > 0$
for every simple object $X$ in $\A$ and, consequently, 
\begin{equation}
\label{dim not 0}
\dim(\A)\neq 0
\end{equation}
for any fusion category $\A$,
see \cite[Theorem 2.3]{ENO1}.  Furthermore, for every simple $X\in \A$ one has 
\begin{equation}
\label{|X| bound}
|X|^2 \leq \FPdim(X)^2
\end{equation}
(see \cite[Proposition 8.21]{ENO1}) and, hence, the categorical dimension of $\A$ is dominated by its
Frobenius-Perron dimension :
\begin{equation}
\label{dim < FPdim}
\dim(\A) \leq  \FPdim(\A).
\end{equation}

\begin{definition}
A fusion category $\A$ over $\mathbb{C}$ is called {\em pseudo-unitary} if 
its categorical and Frobenius-Perron dimensions are equal, i.e.,  $\dim(\A) = \FPdim(\A)$.
\end{definition}

It follows from \eqref{|X| bound} that if $\A$ is pseudo-unitary then $|X|^2 =  \FPdim(X)^2$ for every simple object $X$. 

It is known  \cite{ENO4} that in any fusion category there is a canonical natural tensor isomorphism 
$g_X: X \xrightarrow{\sim} X^{****}$. Let $a_X: X \xrightarrow{\sim} X^{**}$ be a natural isomorphism 
such that $a_{X^{**}} \circ a_X = g_X$ (i.e., $a_X$ is a square root of $g_X$). 
%It is not known if it is possible to choose  $a_X$ to be a tensor isomorphism.
For all $X,\,Y,\, V \in \O(\A)$ let 
\begin{equation}
\label{bXY}
b_{XY}^V : \Hom_A(V,\, X\ot Y) \xrightarrow{\sim} \Hom_A(V^{**} ,\, X^{**}\ot Y^{**})
\end{equation}
be a linear isomorphism such that 
\begin{equation*}
a_X \ot a_Y =\bigoplus_{V \in \O(\A)}\, b_{XY}^V \ot a_V.
\end{equation*}
Note that the source and target of  \eqref{bXY} are canonically isomorphic so that we can regard
$b_{XY}^V$ as an automorphism of $\Hom_\A(V,\, X\ot Y) $.  The natural isomorphism $a_X$
is tensor (i.e., is a pivotal structure) if and only if $b_{XY}^V =\id$ for all $X,\,Y,\, V \in \O(\A)$. Since $a_X$ is a square
root of a tensor isomorphism $g_X$ we see that $(b_{XY}^V)^2 =\id$.  

The integers
\[
N_{XY}^V = \dim_{\mathbb{C}}\, \Hom_\A(V,\, X\ot Y) \qquad \mbox{and} \qquad
T_{XY}^V = \Trace (b_{XY}^V),
\]
where  $\Trace$ denotes the trace of a linear transformation,
satisfy inequality 
\begin{equation}
\label{T < N}
|T_{XY}^V|  \leq N_{XY}^V.
\end{equation} 
The equality $T_{XY}^V  = N_{XY}^V$ occurs if and only $b_{XY}^V =\id$, i.e., if and only if
$a_X$ is a pivotal structure. 

For any $X \in \O(\A)$ let  $d_X = \Tr(a_X)$ then
\[
d_X d_Y =\sum_{V \in \O(\A)}\, T_{XY}^V d_V.
\]  
Furthermore, $|X|^2 =  |d_X|^2$ for every $X\in \O(\A)$.

\begin{proposition}
\label{canonical spherical structure}
A pseudo-unitary fusion category admits a unique spherical structure $a_X: X\xrightarrow{\sim} X^{**}$
with respect to which $d_X = \FPdim(X)$ for every simple object $X$.
\end{proposition}
\begin{proof}
Let $\A$ be a pseudo-unitary fusion category. Let $g_X :  X \xrightarrow{\sim} X^{****}$
be a tensor isomorphism 
and $a_X: X \xrightarrow{\sim} X^{**}$ be its square root as above.   The idea
of the proof is to twist $g_X$ by an appropriate tensor automorphism of the identity 
endofuctor of $\A$ in such a way that  the dimensions corresponding to the square
root of the resulting  isomorphism  become positive real numbers. 

We have  $|d_X| = \FPdim(X)$ for any simple object $X$, therefore,  using \eqref{T < N} we obtain:
\begin{eqnarray*}
\FPdim(X) \FPdim(Y) &=& |d_X d_Y | = |\sum_{V\in \O(\A)}\, T_{XY}^V d_V | \\
 &\leq&  \sum_{V\in \O(\A)}\, N_{XY}^V \FPdim(V) = \FPdim(X) \FPdim(Y),
\end{eqnarray*}
for all $X,\, Y \in \O(\A)$. Hence, the inequality in the above chain is an equality,
i.e., $T_{XY}^V = \pm N_{XY}^V$ and  the ratio $\frac{d_Xd_Y}{d_V}$ is a real number whenever
$N_{XY}^V \neq 0$.  Thus, $\frac{d_X^2d_Y^2}{d_V^2}$
is a positive number whenever $V$ is contained in $X\ot Y$.

The latter property is equivalent to  $\sigma_X:=\frac{|d_X|^2}{d_X^2} \id_X$ being a tensor automorphism of the identity 
endofunctor of $\A$.  Let us replace $g_X$ by  $g_X \circ \sigma_X$ (so it  is still a tensor isomorphism 
$X\xrightarrow{\sim} X^{****}$).   The square  root of the latter is $a_X \circ  \tau_X$, where
$\tau_X = \frac{|d_X|}{d_X} \id_X$. The dimensions corresponding to $a_X \circ  \tau_X$ are now such that
$d_X = |d_X|$, i.e., are positive real numbers. This forces $T_{XY}^V = N_{XY}^V$. 
Thus,  $a_X \circ  \tau_X$ is a spherical structure on $\A$.
\end{proof}

%%%%%%%%%%%%%%%%%%%%%%%%%%%%%%%%%%%%%%%%%%%%%%
%%%%%%%%%%%%%%%%%%%%%%%%%%%%%%%%%%%%%%%%%%%%%%
%%%%%%%%%%%%%%%%%%%%%%%%%%%%%%%%%%%%%%%%%%%%%%
\section{Module categories and categorical Morita equivalence} 
\label{dualmcat}

%%%%%%%%%%%%%%%%%%%%%%%%%%%%%%%%%%%%%%%%%%%%%%
\subsection{Definitions and examples}
\label{sect module}
Let $\A$ be a fusion category.

\begin{definition}
\label{def module cat}
A {\it left module category} over $\A$ is a category $\M$ equipped with an {\em action (or module product) bifunctor} 
$\otimes : \A\times \M\to \M$ along   natural isomorphisms
\begin{equation}
\label{natm}
m_{X,Y,M}: (X\ot Y)\ot M \xrightarrow{\sim} X \ot (Y \ot M), %\qquad X,Y\in \A,\, M\in \M
\end{equation}
and 
\begin{equation}
u_M: \be \ot M   \xrightarrow{\sim}  M,%\qquad  M\in \M,
\end{equation}
called {\em module associativity} and {\em unit constraints}
such   that the  the  following diagrams: 
\begin{equation}
\label{pentagonm}
\xymatrix{
&((X\ot Y)\ot Z)\ot M  \ar[dl]_{a_{X,Y,Z}\otimes \id_M}
\ar[dr]^{m_{X\otimes Y,Z,M}}&\\
 (X \ot (Y \ot Z))\ot M \ar[d]_{m_{X, Y\ot Z, M}} 
&& (X \ot Y) \ot (Z \ot  M) \ar[d]^{m_{X, Y, Z\ot M}} 
\\
 X \ot ((Y\ot Z) \ot M) 
\ar[rr]^{\id_X \ot m_{Y,Z,M}} &&
X \ot (Y \ot (Z\ot M))
}
\end{equation}
and
\begin{equation} 
\xymatrix{
(X \ot \be)\ot M \ar[rr]^{m_{X,\be ,M}} \ar[dr]_{r_X \ot \id_M}  && X \ot (\be \ot M) \ar[dl]^{\id_X \ot u_M}\\
& X \ot M &
}
\end{equation}
commute for all objects $X,Y,Z$ in $\A$ and $M$ in $\M$.
\end{definition}

\begin{definition} 
\label{modfun}
Let $\M$ and $\N$ be two module categories over $\A$ with module associativity
constraints  $m$ and $n$, respectively. An {\em $\A$-module functor}
from $\M$ to $\N$ is a pair $(F,\,s)$, where $F: \M\to \N$  is a functor and
\[
s_{X,M}: F(X\ot M)\to X\ot F(M),%\quad X\in \A,\, M\in \M,
\] 
is a natural isomorphism  such that the following diagrams
\begin{equation}
\label{pentagonf}
\xymatrix{
&& F((X\ot Y)\ot M)   \ar[dll]_{F(m_{X,Y,M})} \ar[drr]^{s_{X\ot Y, M}}  &&\\
F(X\ot (Y\ot M))  \ar[d]_{s_{X,Y\ot M}} &&&& (X\ot Y)\ot F(M)
\ar[d]^{n_{X,Y,F(M)}} \\
X \ot F(Y\ot M) \ar[rrrr]^{\id_X \ot s_{Y,M}}  &&&& X \ot (Y \ot F(M))
}
\end{equation}
and 
\begin{equation}
\label{trianglef}
\xymatrix{
F(\be \ot M)\ar[rr]^{s_{\be,M}} \ar[dr]_{F(l_M) 
}&&\be \ot F(M)
\ar[dl]^{l_{F(M)}}\\ &F(M)&
}
\end{equation}
commute for all $X,Y\in \C$ and $M\in \M$.

A {\em module equivalence} of $\A$-module categories
is an $\A$-module functor  that is an equivalence of
categories. 
\end{definition}

%Clearly, this definition categorifies the notion of a homomorphism (respectively, isomorphism) of modules. 
%See \cite{O} for details and for the definitions of
%$\A$-module functors and their natural transformations.

\begin{definition} 
\label{module nat morph}
A morphism between $\A$-module functors $(F,\,s)$ and  $(G,\,t)$ is a natural transformation $\nu$
from $F$ to $G$ such that the following diagram commutes for any $X\in \A$ and
$M\in \M$:
\begin{equation}
\xymatrix{
F(X\ot M) \ar[rr]^{s_{X,M}}  \ar[d]_{\nu_{X\ot M}}&&  
X\ot F(M) \ar[d]^{\id_X\ot \nu_M} \\
G(X\ot M) \ar[rr]^{t_{X,M}}  &
& X\ot G(M).
}
\end{equation}
\end{definition}

An $\A$-module category is called {\em indecomposable} if it is not
equivalent to a direct sum of two non-trivial $\A$-module categories.

Every $\A$-module category is completely
reducible, i.e., if $\M$ is an $\A$-module category and $\N \subset \M$
is a full $\A$-module subcategory then there exists a full $\A$-module subcategory
$\N' \subset \M$ such that $\M =\N \oplus \N'$.

A typical example of a left $\A$-module category is the category $\A_A$
of  right modules over a separable algebra $A$ in $\A$ \cite{O}.

\begin{example} 
\label{RepG-modules}
Let $G$ be a finite group and let $L\subset G$ be a subgroup, and  
let $\psi\in Z^2(L,\,\kk^\times)$ be a $2$-cocycle on $L$.
By definition, a {\em projective representation} of $L$ on a vector space $V$ 
with the {\em Schur multiplier} $\psi$ is a map $\rho: G \to GL(V)$
such that $\rho(gh) = \psi(g,\,h) \rho(g)\circ \rho(h)$ for all $g,h\in L$.
Let  $\Rep_\psi(L)$ denote the abelian category of projective representations of $L$ with the Schur 
multiplier $\psi$.
The usual tensor
product and associativity and unit constraints determine on $\Rep_\psi(L)$
the structure of a $\Rep(G)$-module category. 
It is known that any indecomposable
$\Rep(G)$-module category is equivalent to  one of this form \cite{O}.
\end{example}

\begin{example}
\label{VecG-modules}
Let $\C= \Vec_G$, where $G$ is a group. 
In this case, a module category $\M$ over $\C$ is an abelian 
category $\M$ with a collection of functors 
\[
F_g: M \mapsto \delta_g\ot M : \M\to \M,
\] 
along with a collection of tensor functor isomorphisms 
\[
\eta_{g,h}: F_g\circ F_h\to F_{gh}, \quad g,h\in G, 
\] 
satisfying the 2-cocycle relation: 
$\eta_{gh,k}\circ \eta_{gh}=\eta_{g,hk}\circ \eta_{hk}$ 
as natural isomorphisms $F_g\circ F_h\circ F_k\xrightarrow{\sim} F_{ghk}$
for all $g,h,k\in G$.

%Recall from Definition~\ref{action} that such data is called an {\em action} of $G$ 
%on $\M$. 
Thus, a module category over $\Vec_G$ is the same thing as an
abelian category with an action of $G$, cf.\ Definition~\ref{def:action}.

Let us describe indecomposable  $\Vec_G$-module categories explicitly. 
In any such category $\M$ the set of simple objects is a transitive $G$-set
$X=G/L$, where a subgroup $L\subset G$ is determined up to a conjugacy.
Let us identify the space of functions $\Fun(G/L,\, \kk^\times)$ with the coinduced 
module $\mbox{Coind}_L^G\, \kk^\times$. 
The $\Vec_G$-module associativity constraint on $\M$ defines a function
\[
\Psi : G \times G \to \mbox{Coind}_L^G\, \kk^\times.
\] 
The pentagon axiom \eqref{pentagonm} says  that $\Psi \in Z^2(G, \, \mbox{Coind}_L^G\, \kk^\times)$.
Clearly, the equivalence class of $\M$ depends only on the  cohomology class of $\Psi$
in $H^2(G, \, \mbox{Coind}_L^G\, \kk^\times)$.  By Shapiro's Lemma the restriction map
\[
Z^2(G, \, \mbox{Coind}_L^G\, \kk^\times) \to Z^2(L,\, \kk^\times) : \Psi \mapsto \psi
\]
induces an isomorphism $H^2(G, \, \mbox{Coind}_L^G\, \kk^\times) \xrightarrow{\sim} H^2(L,\, \kk^\times)$.

Thus, an indecomposable $\Vec_G$-module category is determined by a pair $(L,\, \psi)$,
where $L\subset G$ is a subgroup  and $\psi \in H^2(L,\, \kk^\times)$. Let $\M(L,\, \psi)$
denote the corresponding category. 
\end{example}

\begin{remark}
\label{Not a coincidence}
Note that indecomposable $\Rep(G)$-module categories in Example~\ref{RepG-modules}
and indecomposable $\Vec_G$-module categories in  Example~\ref{VecG-modules}
are parameterized by the same data. We will see in Section~\ref{sect duality}
that this is not merely a coincidence.
\end{remark}

\begin{example}
\label{VecGw-modules}
This is a generalization of Example~\ref{VecG-modules}. Here we describe
indecomposable module categories over pointed fusion categories. Recall
that the latter categories are equivalent to $\Vec_G^{\omega}$ for some
finite group $G$ and  a $3$-cocycle  $\omega \in Z^3(G, \, \kk^\times)$.

Equivalence classes of indecomposable right $\Vec_G^{\omega}$-module categories 
correspond to  pairs
$(L, \, \psi)$, where $L$ is a subgroup of $G$ such that
$\omega|_{L \times L \times L}$ is cohomologically trivial
and $\psi \in C^2(L, \, \kk^\times)$ is a $2$-cochain satisfying 
$\delta^2\psi = \omega|_{L \times L \times L}$. The corresponding
$\Vec_G^{\omega}$-module category is constructed as follows.
Given a pair $(L, \, \psi)$ as above define an algebra 
\begin{equation}
\label{eqR}
A(L, \psi) =\bigoplus_{a\in L} \delta_a
\end{equation}
in $\Vec_G^\omega$ with the multiplication
\begin{equation}
\label{algebra R}
\bigoplus_{a,b\in L}\, \psi(a,\,b) \id_{\delta{ab}} :  
A(L,\, \psi) \ot A(L,\, \psi) \to A(L,\, \psi).
\end{equation}
Let $\M(L,\, \psi)$ denote the category of left $A(L,\, \psi)$-modules in $\Vec_G^\omega$.
Any $\Vec_G^\omega$-module category is equivalent to some $\M(L,\, \psi)$.
%The rank of $\M(L,\, \psi)$ is equal to the index of $L$ in $G$.
%Two $\Vec_G^\omega$-module categories $\M(L,\, \psi)$ and $\M(L',\, \psi')$ 
%are equivalent if and only if there is $g\in G$ such that
%$L'=gLg^{-1}$  and $\psi$ and the $g$-conjugate of $\psi'$ differ by
%a coboundary.
\end{example}

\begin{remark}
Two $\Vec_G^\omega$-module categories $\M(L,\, \psi)$ and $\M(L',\, \psi')$ 
are equivalent if and only if there is $g\in G$ such that
$L'=gLg^{-1}$  and  and $\psi'$ is cohomologous to $\psi^g$
in $H^2(L',\, \kk^\times)$,
where $\psi^g(x,\,y):= \psi(gxg^{-1},\, gyg^{-1})$ for all $x,\,y\in L$. Here we abuse notation and identify
$\psi$ and $\psi'$ with cocylces representing them.
\end{remark}

\begin{example}
Let $H$ be a semisimple Hopf algebra. Example~\ref{VecG-modules} was generalized in 
\cite{AMo} where it was shown that  indecomposable
$\Rep(H)$-module categories are classified by left
$H$-comodule algebras that are $H$-simple from the right and with the trivial space of coinvariants.
This is another generalization of Example~\ref{VecG-modules}.

Here is an other generalization of Example~\ref{VecG-modules}.

Let $H$ be a semisimple Hopf algebra and let $B \subset A$ be a left faithfully flat $H$-Galois extension with 
$B$ semisimple. Let $\M_B$ and $\M^H$ denote  the fusion categories of right $B$-modules and right $H$-comodules, respectively. 
Recall that the category of right Hopf $(H,\, A)$-modules $\M^H_A$ is by deÞnition the category 
of right $A$-modules over $\M^H$. By Schneider's structure theorem \cite{Schn}
the functor 
\[
\M_B \to (\M^H )_A :  M \mapsto  M \ot_B A, 
\]
is a category equivalence with inverse $M \to  M^{co\,H}$, where $M^{co\, H}$ denotes
the subspace of coinvariants.
So $\M_B$ has an $\M^H$-module category structure.
\end{example}

%%%%%%%%%%%%%%%%%%%%%%%%%%%%%%%%%%%%%%%%%%%%%%%%%%%
\subsection{Duality for fusion categories and categorical Morita equivalence} 
\label{sect duality}

Let $\A$ be a fusion category and  let $\M$ be an indecomposable left $\A$-module category.
The category $\A^*_\M$ of $\A$-module endofunctors of $\M$ has a tensor category
structure with a tensor product given by the composition of functors and  the unit object  
being the identity functor. It is also a rigid category with the duals of a functor
being its adjoints (thanks to the rigidity of $\A$  an adjoint of an $\A$-module functor has a natural  
$\A$-module functor structure).

It was shown in \cite[Theorem 2.18 ]{ENO1} that $\A^*_\M$
is a fusion category. The category $\A^*_\M$  is called the {\em dual} category of $\A$
with respect to $\M$. Furthermore, $\M$ has a natural structure of an $\A^*_\M$-module category and there is
a canonical tensor equivalence 
\[
(\A^*_\M)^*_\M \cong \A.
\] 
The Frobenius-Perron dimension is invariant under duality, i.e., 
\begin{equation}
\label{FPdim duality}
\FPdim(\A) =\FPdim(\A^*_\M).
\end{equation}

\begin{example}
Let $A$ be a separable  algebra in $\A$ and let $\M$ be the category of right $A$-modules in $\A$.
Then $(\A^*_\M)^\op$ is tensor equivalent to the category of $A$-bimodules in $\A$.  The tensor product of
the latter category is $\ot_A$ and the unit object is the regular $A$-module. 
\end{example}

\begin{definition}
Let $\A,\, \B$ be fusion categories. We say that $\A$ and $\B$ are {\em categorically
Morita equivalent} if there is an $\A$-module category $\M$ such that $\B\cong (\A^*_\M)^\op$.
\end{definition}

It was shown in \cite{Mu4} that categorical Morita equivalence is indeed an equivalence relation.

\begin{remark}
\label{integral Morita}
The class of integral fusion categories (see Definition~\ref{def:integral}) 
is closed under Morita equivalence  \cite[Theorem 8.35]{ENO1}. 
\end{remark} 

Given a pair of left $\A$-module categories let $\Fun_\A(\M,\, \N)$ denote the category of $\A$-module
functors from $\M$ to $\N$. In particular, $\A^*_\M = \Fun_\A(\M,\, \M)$.  The assignment
\[
\N \mapsto \Fun_\A(\M,\, \N)
\] 
defines a $2$-equivalence between the $2$-category of left $\A$-module categories and
that of right $\A^*_\M$-module categories \cite{EO, Mu-I}.  
This explains the observation we made in Remark~\ref{Not a coincidence}.

Below we give examples of  categorical Morita equivalence.

\begin{example}
Any fusion category $\A$ can be viewed as the regular left module category over itself.  It is easy to see
that in this case left $\A$-module functors are precisely functors of the right multiplication by objects of $\A$,
whence $\A^*_\A = \A^\op$.
\end{example}

\begin{example}
Any fusion category $\A$ can be viewed as an $\A \bt \A^\op$-module category with the actions
of $\A$   given by left and right tensor multiplications. The dual category
$(\A \bt \A^\op)^*_\A$ is the center of $\A$, see  Section~\ref{Center section} below. 
\end{example}

\begin{example}
\label{VecG dual}
Let $G$ be a finite group and let $\A =\Vec_G$ be the category of $G$-graded
vector spaces. The category $\Vec$ is a 
$\Vec_G$-module category via the forgetful tensor functor $\Vec_G \to \Vec$.
Let us determine the dual category  $(\Vec_G)^*_{\Vec}$.  Unfolding the definition 
of a module functor, we see that a $\Vec_G$-module
endofunctor   $F: \Vec\to \Vec $ is determined by a vector space $V:= F(\kk)$ and a collection
of isomorphisms 
\[
\pi_g \in \Hom_{\Vec}(F(\delta_g \ot k),\,\delta_g\ot F(\kk)) =\End_\kk(V),\, g\in G.
\]
It follows from axiom~\eqref{pentagonf}  in Definition~\ref{modfun}  of module functor
that the map 
\[
g \mapsto \pi_g : G \to GL(V)
\]
is a representation of $G$ on $V$. 
Conversely, any such representation determines a $\Vec_G$-module
endofunctor of $\Vec$.  It is easy to check that homomorphisms of representations
are precisely morphisms between the corresponding module functors.
Thus, $(\Vec_G)^*_{\Vec} \cong \Rep(G)$, i.e., the categories $\Vec_G$ and $\Rep(G)$ 
are categorically  Morita equivalent.
\end{example}

\begin{example}
This is a generalization of the previous example.
Let $H$  be a finite-dimensional Hopf algebra.  The fiber functor $\Rep(H)\to \Vec$
makes $\Vec$ a $\Rep(H)$-module category and ${\Rep(H)^*}_{\Vec} \cong \Rep(H^*)$.  
Thus, $\Rep(H)$ and $\Rep(H^*)$ are categorically  Morita equivalent. This means that
categorical duality extends the notion of Hopf algebra duality. 
\end{example}

\begin{example}
\label{equiv dual grad}
Let $G$ be a finite group and let $g\mapsto T_g$ be an action of $G$  on a fusion category $\A$. 
The forgetful tensor functor  $\A^G \to \A$
turns $\A$ into an $\A^G$-module fusion category. The dual category $(\A^G)^*_\A$ is the {\em crossed product}
category  $\A \rtimes  G$ defined as follows.   As an Abelian category  $ \A \rtimes  G = \A \bt \Vec_G$.
The tensor product is given by
\begin{equation}
(X \bt \delta_g) \ot (Y \bt \delta_h): = (X \ot T_g(Y)) \bt \delta_{gh},\qquad X,Y \in \A,\quad g,h\in G.
\end{equation}
The unit object is $\be \bt \delta_e$ and the associativity and unit constraints
come from those of $\A$.

Note that $\C \rtimes G$ is a $G$-graded fusion category,
\[
\A \rtimes G =\bigoplus_{g\in G}\, (\A \rtimes G)_g,\qquad \mbox{ where }     
(\A \rtimes G)_g= \C \ot (\be \bt \delta_g).
\]
In particular, $\A \rtimes G$ contains $\A =\A \ot (\be \bt \delta_e)$ as a fusion subcategory.

In the case when $\A =\Vec$ we recover the duality of Example~\ref{VecG dual}.
\end{example}

%%%%%%%%%%%%%%%%%%%%%%%%%%%%%%%%%%%%%%%%%%%%%%%%%%%
\subsection{(Weakly) group-theoretical fusion categories} 
\label{sect group-theor}

In this Section we use categorical Morita equivalence to introduce two classes of categories 
important for  classification of fusion categories of integer  
Frobenius-Perron dimension. 

\begin{definition}
\label{def:group-theoretical}
A fusion category is called {\em group-theoretical} if it is categorically Morita equivalent 
to a pointed fusion category. 
\end{definition}

In other words, a group-theoretical fusion category  is equivalent to  $(\Vec_G^\omega)^*_{\M(L,\, \psi)}$,
where $\M(L,\, \psi)$ is  the $\Vec_G^\omega$-module category from Example~\ref{VecGw-modules}. 

\begin{example}
\label{gt example}
The category $(\Vec_G^\omega)^*_{\M(L,\, \psi)}$ can be described quite  explicitly  
in terms of finite groups and their cohomology as the category
of  $A(L, \psi)$-bimodules in $\Vec_G^\omega$, where $A(L, \psi)$ is the algebra
introduced in \eqref{eqR} (see \cite[Proposition 3.1]{O2}). 

For example, simple objects of  $(\Vec_G^\omega)^*_{\M(L,\, \psi)}$ can be described as follows.
  
For any $g\in G$ the group $L^g:= L\cap g L g^{-1}$ 
has a well-defined $2$-cocycle
\begin{eqnarray*}
%\begin{split}
\psi^g(h,h') :&=&  \psi(h,h') \psi(g^{-1}h'^{-1}g,\, g^{-1}h^{-1}g)
\omega(hh'g,\, g^{-1}h'^{-1}g,\, g^{-1}h^{-1}g)^{-1} \\
& & \times \,  \omega(h,\,h',\,g) \omega(h,\, h'g,\, g^{-1}h'^{-1}g),\qquad h,h'\in L^g.
%\end{split}
\end{eqnarray*}
%Let $\{g_i\}_{i\in H_1\backslash G  / H_2}$ be a set of representatives of 
%two-sided $(H_1,\, H_2)$-cosets in $G$.
One can  check that irreducible $A(L, \psi)$-bimodules in $\Vec_G^\omega$
are parameterized by  pairs $(Z,\, \pi)$, where $Z$ is a  double $L$-coset 
in $G$ and $\pi$ is an irreducible projective representation 
of $L^{g}$ with the Schur multiplier $\psi^{g},\, g\in Z$. 
\end{example}

\begin{example}
Let $G,\,L$ be groups and let $H$ be a semisimple Hopf algebra that fits into an exact
sequence 
\[
1 \to \kk^G \to H \to \kk L \to 1,
\]
where $\kk^G$ is the commutative Hopf algebra of functions on $G$ and $\kk L$ is the cocommutative
group Hopf algebra of $L$ (i.e., $H$ is an extension of $\kk L$ by $\kk^G$). 
It was shown  in \cite{Nat} that $\Rep(H)$ is a group-theoretical
fusion category. 
\end{example}

\begin{remark}
The class of group-theoretical fusion categories is not closed under equivariantizations \cite{GNN, Ni}.
In particular, the class of Hopf algebras  with group-theoretical representation
categories is not closed under Hopf algebra extensions. 

The smallest  example of  a semisimple Hopf algebra whose representation category is not  
group-theoretical  has dimension $36$ \cite{Ni}. 
\end{remark}

\begin{remark}
In view of Remark~\ref{integral Morita}, group-theoretical categories are integral.
\end{remark}

Recall that the notion of a nilpotent fusion category was defined in Section~\ref{GT construction sect}. 
The following definition was given in \cite{ENO2}. 

\begin{definition}
\label{def:weakly group-theoretical}
A fusion category is {\em weakly group-theoretical} if it is categorically Morita equivalent 
to a nilpotent fusion category.  A fusion category is {\em solvable} if it is categorically Morita equivalent 
to a cyclically nilpotent fusion category.
\end{definition}

Here is a list of properties of solvable categories (see \cite[Proposition 4.4]{ENO2}).

\begin{proposition}
\label{basicp2-bis}
\begin{enumerate}
\item[(i)] The class of solvable categories is closed
under taking extensions and equivariantizations by solvable
groups, Morita equivalent categories, tensor products, 
subcategories and component categories of 
quotient categories. 
\item[(ii)]
The categories $\Vec_{G,\omega}$ and $\Rep(G)$ are solvable
if and only if $G$ is a solvable group.
\item[(iii)]
A braided nilpotent fusion category is solvable.
\item[(iv)] 
A solvable fusion category $\A\ne \Vec$ contains a nontrivial
invertible object. 
\end{enumerate}
\end{proposition}

\begin{lemma}
\label{extension Morita}
Let $G$ be a finite group, let $\A$ be a $G$-extension  of a fusion category $\A_0$, and let
$\B_0$ be a fusion category Morita equivalent to $\A_0$. There exists
a $G$-extension $\B$ of $\B_0$ which is Morita equivalent to $\A$. 
\end{lemma}
\begin{proof}
The proof is taken from \cite[Lemma 3.4]{ENO2}. Let $A$ be an algebra in $\A_0$
such that $\B_0$ is equivalent to the category of $A$-bimodules in  $\A_0$.  Let $\B$ be the category
of $A$-bimodules in  $\A$ (we can view $A$ as an algebra in $\A$ since $\A_0 \subset \A$).
Then  $\B$ inherits the $G$-grading, thanks to $A$  being in the trivial component of the $G$-graded
fusion category $\A$.  By construction, $\B$ is categorically Morita equivalent to $\A$. 
\end{proof}

\begin{proposition}
\label{WGT closure}
The class of weakly group-theoretical fusion closed is closed under extensions and equivariantizations.
\end{proposition}
\begin{proof}
In view of Example~\ref{equiv dual grad} it is enough to prove the assertion about extensions.
Let $\A$ be a $G$-extension of a weakly group-theoretical fusion category $\A_0$.  Let $\B_0$
be a nilpotent fusion category Morita equivalent to $\A_0$.  Then by Lemma~\ref{extension Morita}
there exists a nilpotent category $\B$  Morita equivalent to $\A$, i.e., $\A$ is weakly group-theoretical.
\end{proof}

\begin{remark}
Proposition~\ref{WGT closure} and the fact that the class of weakly group-theoretical categories
is closed under taking subcategories and  quotient categories were proved in 
\cite[Proposition 4.1]{ENO2}.
\end{remark}

\begin{remark}
Since the Frobenius-Perron dimension of a fusion category is invariant under
categorical Morita equivalence, cf. \eqref{FPdim duality},  we have  $\FPdim(\A) \in \mathbb{Z}$
for every weakly group-theoretical fusion category $\A$.
\end{remark}

\begin{remark}
Let $\A$ be a weakly group-theoreticsl fusion category.  The following Frobenius property
of $\A$ was established in \cite[Theorem 1.5]{ENO2}:  for every simple object $X$ of $\A$ the ratio
$\FPdim(\A)/ \FPdim(X)$ is an algebraic integer.

\end{remark}

\begin{remark}
Because of recursive nature of the definition of a nilpotent fusion category  it is usually not
possible  to describe weakly group-theoretical  categories as explicitly as group-theoretical ones,
cf.\ Example~\ref{gt example}. 
On the other hand,  there is a classification of module categories over a given graded
fusion category in terms of module categories over its trivial component \cite{Ga, MM}.
So, in principle, weakly group-theoretical fusion categories can be described in terms
of finite groups and their cohomology. 
\end{remark}

\begin{remark}
The approach to classification  of fusion categories used in this paper  consists of showing
that categories of a given Frobenius-Perron dimension  are weakly group-theoretical. 
We {\em do not} attempt to classify weakly group-theoretical categories of an
arbitrary finite dimension (indeed, this would include, as a special case,  classification
of finite groups).
\end{remark}

%%%%%%%%%%%%%%%%%%%%%%%%%%%%%%%%%%%%%%%%%%%%%%
%%%%%%%%%%%%%%%%%%%%%%%%%%%%%%%%%%%%%%%%%%%%%%
%%%%%%%%%%%%%%%%%%%%%%%%%%%%%%%%%%%%%%%%%%%%%%
\section{Braided fusion categories}

The notion of a braiding was introduced by A.~Joyal and  R.~Street in \cite{JS}.
%%%%%%%%%%%%%%%%%%%%%%%%%%%%%%%%%%%%%%%%%%%%%%%%%%%%%%%%%
\subsection{Definitions and examples}
\label{def braid sect}

\begin{definition}\label{brcat}
A {\em braiding} on a fusion category $\C$ is 
a natural isomorphism 
\[
c_{X,Y}: X \ot Y \xrightarrow{\sim}  Y \ot X,\qquad X,\,Y\in \C, 
\]
called a {\em commutativity constraint}, such that the following hexagon diagrams 
\begin{equation}
\label{hexagon}
\xymatrix{
& X \ot (Y \ot Z) \ar[rr]^{c_{X, Y\ot Z}} & &
(Y \ot Z) \ot X  \ar[dr]^{a_{Y,Z,X}} & \\
(X\ot Y)\ot Z \ar[ur]^{a_{X,Y,Z}}
\ar[dr]_{c_{X,Y}\ot \id_Z} & &  &  &Y \ot (Z\ot X) \\
& (Y\ot X)\ot Z \ar[rr]_{\alpha_{Y, X, Z}} & &  Y \ot (X\ot Z)
\ar[ur]_{\id_Y  \ot c_{X, Z}} & 
}
\end{equation}
and
\begin{equation}
\label{hexagon'}
\xymatrix{
& (X \ot Y) \ot Z \ar[rr]^{c_{X\ot Y, Z}} & &
Z \ot (X \ot Y)  \ar[dr]^{a^{-1}_{Z,X,Y}} & \\
X\ot (Y\ot Z) \ar[ur]^{a^{-1}_{X,Y,Z}}
\ar[dr]_{\id_X \ot c_{Y,Z}} & &  &  & (Z\ot X) \ot Y\\
& X \ot (Z\ot Y) \ar[rr]_{a^{-1}_{X, Z, Y}} & &  (X\ot Z) \ot Y
\ar[ur]_{c_{X, Z}\ot \id_Y} & 
}
\end{equation}
are commutative  for all objects $X,\,Y,\,Z$ in ${\C}$. 
\end{definition}
%The above axioms are called Hexagon axioms.

For a braided fusion category $\C$ with the braiding $c_{X,Y}: X\ot Y \xrightarrow{\sim} Y\ot X$
let $\C^\rev$ denote the {\em reverse} category that coincides with $\C$ as a fusion category
and has braiding $\tilde{c}_{X,Y}:= c_{Y,X}^{-1}$.

\begin{definition}
\label{Braided tensor functor}
Let $\C^1$ and $\C^2$ be braided tensor categories whose braidings are denoted 
$c^1$ and $c^2$,
respectively. A tensor functor $(F,\, J)$ from $\C^1$ to $\C^2$ is called {\em braided}
if the following diagram commutes:
\begin{equation}
\label{tensor hexagon}
\xymatrix{
F(X) \ot F(Y) \ar[rrr]^{c^2_{F(X),F(Y)}} \ar[d]_{J_{X,Y}}  &&& F(Y)  \ot F(X) 
\ar[d]^{J_{Y,X}}  \\
F(X \ot Y) \ar[rrr]^{F(c^1_{X,Y})} &&& F(Y \ot X)
}
\end{equation}
for all object $X,Y$ in $\C^1$.
\end{definition}

Note that a tensor functor is a functor with an additional {\em structure}. 
For a tensor functor to be braided is a {\em property}.

Let $G$ be an Abelian group.  By a {\em quadratic form} on $G$
(with values in $\kk^\times$) we will mean a map $q: G \to \kk^\times$
such that $q(g)=q(g^{-1})$ and  the symmetric function
\begin{equation}
\label{bq}
b(g,\,h) := \frac{q(gh)}{q(g)q(h)}
\end{equation}
is bimultiplivative, i.e.,  $b(g_1g_2,\, h) =b(g_1,\, h) b(g_2,\, h)$
for all $g,\,g_1,\, g_2,\, h\in G$.
We will say that $q$ is {\em non-degenerate} if the associated bicharacter
$b$ is non-degenerate.

The simplest way to construct a quadratic form on $G$ is to start with a
bicharacter $B : G \times G \to \kk^\times$ and set
\begin{equation}
\label{easy q}
q(g) := B(g,g),\, g\in G.  
\end{equation}
%Not every quadratic form
%can be represented like this. Counterexample: $G =\mathbb{Z}/2\mathbb{Z}$
%with $q(n) = i^{n^2}$.

\begin{definition}
A {\em pre-metric group} is a pair $(G,\, q)$ where $G$ is a finite Abelian
group and $q:G\to \kk^\times$ is a quadratic form. A {\em metric group}
is a pre-metric group such that $q$ is non-degenerate.
\end{definition}

The relation between braided fusion categories and pre-metric groups is as follows.
Let $\C$ be a pointed braided fusion category. Then $\C =\Vec_G^\omega$
for some {\em Abelian} group $G$.  
Define a map $q: G \to \kk^\times$ by
\begin{equation*}
q(g) = c_{\delta_g,\, \delta_g}  \in \Aut (\delta_g \ot \delta_g) =\kk^\times.
\end{equation*}
It is easy to see that $q:G \to \kk^\times$ is a quadratic form. We thus have
a functor:
\begin{equation*}
F: (\mbox{pointed braided fusion categories}) \to (\mbox{pre-metric groups}). 
\end{equation*}
It was shown by Joyal and Street in \cite{JS}  that this functor is an equivalence.
Under this equivalence, braided tensor functors correspond to orthogonal
(i.e., quadratic form preserving) homomorphisms.

We will denote by $\C(G,\, q)$ the braided fusion category associated 
to the pre-metric group $(G,\, q)$. 

\begin{remark}
When $q$ is determined by a bicharacter $B$ as in \eqref{easy q} we  have
$\C(G,\, q) =\Vec_G$ as a fusion category with  the braiding given by
\begin{equation*}
c_{\delta_g,\, \delta_h} = B(g,\, h) \id_{\delta_{gh}}. 
\end{equation*} 
\end{remark}

\begin{definition}
A {\em quasi-triangular structure} on a Hopf algebra $H$ is an
invertible element $R \in H \otimes H$ such that
for all $x \in H$, 
\begin{equation}
R \Delta (x) = \Delta^{op}(x) R \label{comul},
\end{equation}
where $ \Delta^{op}$ denotes the opposite comultiplication,
and the following relations are satisfied (in $H  \otimes H \otimes H$): 
\begin{eqnarray}
(\Delta \otimes \id)(R) &=& R^{13} R^{23} \label{hexone} \\
(\id \otimes \Delta)(R) &=& R^{13} R^{12} \label{hextwo}
\end{eqnarray}
A Hopf algebra $H$ equipped with a quasi-triangular structure
is said to be a {\em quasi-triangular Hopf algebra}.
\end{definition}

Here 
for $R =\sum_i\, r_i\ot r_i^\prime$ we write $R^{12} =\sum_i \, 
r_i\ot r_i^\prime \ot 1 \in H \otimes H \otimes H$, etc. 

Given a semisimple quasi-triangular Hopf algebra $H$ one turns $\Rep(H)$ into
a braided fusion category by setting
\begin{equation}
\label{braiding from QT}
c_{V\ot W}  : V \ot W \xrightarrow{\sim} W\ot V: v\ot w \mapsto R^{21}(w\ot v),
\end{equation}
for all representations $V,\, W$ of $H$ and $v\in V,\, w\in W$.

Note that the axiom \eqref{comul} means that the map \eqref{braiding from QT}
is a morphism in $\Rep(H)$ and axioms \eqref{hexone} and \eqref{hextwo}
are equivalent to  $c$ satisfying the hexagon axioms \eqref{hexagon} and \eqref{hexagon'}. 
Conversely,
any braiding on $\Rep(H)$ is determined by a quasi-triangular structure.

%%%%%%%%%%%%%%%%%%%%%%%%%%%%%%%%%%%%%%%%%%%%%%%%%%%%%%%%%
\subsection{Symmetric and Tannakian subcategories}

A braided fusion category $\C$ is called {\em symmetric} if  $c_{Y,X}c_{X,Y} =\id_{X\ot Y}$ for all objects $X,Y \in \C$; 
in this case the braiding $c$ is also called {\em symmetric}. 

\begin{example}
The category $\Rep(G)$ of representations of a finite group $G$ equipped with its standard
symmetric braiding $c_{X,Y}(x\otimes y):=y\otimes x$
is an example of a symmetric fusion category. Deligne \cite{De}
proved that any symmetric fusion category is equivalent to a ``super" generalization of $\Rep(G)$. 
Namely, let $G$ be a finite group and let $z\in G$
be a central element such that $z^2=1$. Then the fusion category $\Rep(G)$ has a braiding
$c'_{X,Y}$ defined as follows:
\begin{equation*}
c'_{X,Y}(x\otimes y)=(-1)^{mn}y\otimes x \,\mbox{ if }\,
x\in X, \; y\in Y, \; zx=(-1)^{m}x,\; zy=(-1)^{n}y.
\end{equation*}
Let $\Rep(G,\,z)$ denote the fusion category $\Rep(G)$ equipped with the above braiding.
Equivalently, $\Rep(G,\,z)$ can be described as a full subcategory of the category of super-representations of $G$; 
namely, $\Rep(G,\,z)$ consists of those super-representations $V$ on 
which $z$ acts by the {\em parity automorphism\,} 
(i.e., $zv=v$ if $v\in V$ is even and $zv=-v$ if $v\in V$ is odd). 

For example, let $G=\mathbb{Z}/2\mathbb{Z}$ and $z$ be the nontrivial element of $G$. Then
$\Rep(G,\,z)$ is the category $\sVec$ of super-vector spaces.
\end{example}

A symmetric fusion category $\C$ is said to be {\em Tannakian} if 
there exists a finite group $G$ such that $\C$ is equivalent to $\Rep(G)$ as a braided fusion
category. It is proved in \cite{De} that $\C$ is Tannakian if and only if it 
admits a {\em braided fiber functor}, i.e., a  braided tensor functor $\C\to\Vec$. 

The canonical fiber functor $\Rep(G)\to \Vec$ is nothing but the functor forgetting
the $G$-module structure.

Note that for a symmetric category $\C$ its dimension $\FPdim(\C)$ is always an integer (more precisely, $\FPdim(\Rep(G,\, z)) =|G|$).
In particular, if $\FPdim(\C)$ is {\em odd} then $\C$ is automatically Tannakian.

\begin{remark}
\label{non-triv tannakian}
Let $\C=\Rep(G,\, z)$ be a symmetric category. Then $\Rep(G/\la z \ra)$ is a Tannakian subcategory of  $\C$.
In particular, a symmetric fusion  category $\C \not\cong \sVec$ contains a non-trivial Tannakian subcategory.
\end{remark}

\begin{example}
The pointed braided fusion category $\C(G,\, q)$ associated to the pre-metric group
$(G,\, q)$  is symmetric if and only if $q$ is a homomorphism.   The category $\C(G,\, q)$  is Tannakian if and
only if $q=1$.
\end{example}

%%%%%%%%%%%%%%%%%%%%%%%%%%%%%%%%%%%%%%%%%%%%%%%%%%%%%%%%%
\subsection{The Drinfeld center construction}
\label{Center section}

We now give a construction which assigns to every fusion category $\A$ a
braided fusion category $\Z(\A)$, called the {\em center} of $\A$. 

Explicitly, the objects of $\Z(\A)$ are pairs $(X,\, \gamma)$, where $X$ is an object of $\A$ and
\begin{equation}
\label{gamma}
\gamma = \{ \gamma_{V} : V\ot  X \xrightarrow{\sim}  X \ot V \}_{ V \in \A}
\end{equation}
is a natural family of isomorphisms, called {\em half-braidings}, 
making the following diagram commutative:
\begin{equation}
\label{central object}
\xymatrix{
& V \ot (X \ot U) \ar[rr]^{a_{V,X,U}^{-1}} & &
(V \ot X) \ot U  \ar[dr]^{\gamma_{V}\ot \id_U} & \\
V\ot (U\ot X) \ar[ur]^{\id_V\ot \gamma_{U}}
\ar[dr]_{a_{V,U,X}^{-1}} & &  &  &(X \ot V) \ot U \\
& (V\ot U)\ot X \ar[rr]_{\gamma_{V\ot U}} & & X\ot (V\ot U)
\ar[ur]_{a_{X,V,U}^{-1}} & }.
\end{equation}

The center has a canonical braiding given by
\begin{equation}
\label{braiding}
c_{(X, \gamma),\,(X',\gamma')} = \gamma_{X'} :
 (X,\, \gamma) \otimes (X',\, \gamma') \xrightarrow{\sim}
(X',\, \gamma') \otimes (X,\, \gamma).
\end{equation}
Furthermore, there is an obvious forgetful tensor functor:
\begin{equation}
\label{ZC to ZDC} 
F: \Z(\A) \mapsto \A: (X,\,\gamma) \mapsto X.
\end{equation}

We have
\begin{equation}
\label{dim center}
\FPdim(\Z(\A)) =\FPdim(\A)^2 \quad \mbox{and} \quad \dim(\Z(\A)) =\dim(\A)^2, 
\end{equation}
where the  Frobenius-Perron dimension  $\FPdim(\A)$ and categorical dimension $\dim(\A)$ 
were introduced in \eqref{FPdim def} and \eqref{dim def}.

Let $\C$ be a braided fusion category with braiding $c_{X,Y}: X\ot Y \xrightarrow{\sim} Y \ot X$.
There are natural braided embeddings $\C,\, \C^\rev \hookrightarrow \Z(\C)$ given by 
\[
X \mapsto ( X,\,  c_{-, X}) \quad \mbox{and} \quad   X \mapsto ( X,\,  \tilde{c}_{-, X}).
\]
They combine into a braided tensor functor
\begin{equation}
\label{G:CC-->ZC} 
G:\C \boxtimes \C^\rev\to \Z (\C).
\end{equation}
We say that $\C$ is {\em factorizable}  if the functor \eqref{G:CC-->ZC} is an equivalence.

\begin{example}
\label{hyperbolic form}
Let $A$ be a finite  Abelian group. There is canonical quadratic form 
\begin{equation}
q: A\oplus A^* \to \kk^\times : (a,\, \phi) \mapsto \phi(a),\qquad a\in A,\, \phi\in A^*.
\end{equation}
We have $\Z(\Vec_A)\cong \C(A\oplus A^*,\, q)$.
\end{example}

\begin{example}
\label{ZVecG}
More generally, let $G$ be a finite group. The category $\Z(\Vec_G)$ is equivalent
to the category  of $G$-equivariant vector bundles on $G$, cf.\ Example~\ref{equiv examples}(iii). 
Here $G$ acts on itself by conjugation. Explicitly, a $G$-equivariant vector bundle is
a graded vector space $V =\oplus_{g\in G}\, V_g$ along with a collection
of isomorphisms $\phi_{x,g} : V_g \to V_{xgx^{-1}}$ satisfying
the compatibility condition 
\[
\phi_{x,\, ygy^{-1}} \phi_{y ,g} = \phi_{xy, g},
\]
for all $x,\, y,\,g\in G$.   Note that $\Z(\Vec_G)$ contains a subcategory $\Rep(G)$
as the bundles supported on the identity element of $G$.
\end{example}

\begin{example}
Let $H$ be a semisimple Hopf algebra. Then 
\[
\Z(\Rep(H)) \cong \Rep(D(H)),
\] 
where  $D(H)$ is the Drinfeld double of $H$ \cite{Ka}. 
\end{example}

%%%%%%%%%%%%%%%%%%%%%%%%%%%%%%%%%%%%%%%%%%%%%%%%%%%%%%%%%
\subsection{Ribbon fusion categories and traces}

\begin{definition}
A {\em pre-modular} fusion category is a braided fusion category
equipped with a spherical structure. 
\end{definition}

Below we give an equivalent description of pre-modular categories.

\begin{definition}
\label{def ribbon} 
A {\em twist}  (or a {\em balanced transformation}) on a braided
fusion category $\C$ is $\theta\in \Aut(\id_\C)$
such that 
\begin{equation}
\label{twist}
\theta_{X\ot Y}  = (\theta_X \ot \theta_Y) c_{Y,X} c_{X,Y} 
%\qquad \mbox{ and } \qquad (\theta_X)^* =\theta_{X^*}
\end{equation}
for all $X,\,Y\in \C$. A twist is called {\em a ribbon structure} if $(\theta_X)^* =\theta_{X^*}$.
A   fusion category with a ribbon structure is called a {\em ribbon} category.
\end{definition} 

\begin{remark}
The notion of a ribbon structure can be understood as a 
generalization of the notion of quadratic form.  
Indeed, let $G$ be a finite Abelian group and $b: G \times G \to \kk^\times$
be a bilinear form. As explained in Section~\ref{def braid sect}, it defines a braiding on $\C=\Vec_G$.
The corresponding quadratic form defines a ribbon structure on $\C$:
$$
\theta_{\delta_x}= b(x,\, x) \id_{\delta_x},\qquad x\in G. 
$$
\end{remark}

%\begin{remark}
%Let $\theta_1,\, \theta_2$ be two ribbon structures on $\C$. Then $\psi:= \theta_1 \theta_2^{-1}$
%is a tensor automorphism of $\id_\C$ such that $\psi^2 =\id$. 
%\end{remark}

Let us define a natural transformation $u_X: X \to X^{**}$ as the composition
\begin{equation}
\label{u}
X \xrightarrow{\id_X \ot \coev_{X^*}} X \ot X^* \ot X^{**} \xrightarrow{c_{X,X^*} \ot \id_{X^{**}}}
X^* \ot X \ot X^{**}  \xrightarrow{\ev_{X}\ot \id_{X^{**}}} X^{**}.
\end{equation}

Then $u_X$ is an isomorphism satisfying  the following balancing property:
\begin{equation}
\label{balance}
u_X \ot u_Y = u_{X\ot Y}\, c_{Y,X}c_{X,Y}
\end{equation}
or all $X,\,Y\in \C$.

Clearly,  any natural isomorphism $\psi_X:X\simeq X^{**}$ in
a braided fusion category $\C$ can be written as 
\begin{equation} 
\label{deltatheta}
\psi_X= u_X \theta_X.
\end{equation}
for some $\theta\in\Aut (\id_\C )$ It follows from \eqref{balance} that
$\psi$ is a tensor isomorphism (i.e., a pivotal structure on $\C$) if and only if $\theta$ is a twist.

The above pivotal structure is  spherical if an only if the corresponding twist $\theta=\psi u^{-1}$
is a ribbon structure. Thus, a ribbon fusion category is the same thing as a pre-modular category.

%The trace \eqref{old friend} is related to left and right traces $\Tr^L$ and $\Tr^R$ from \eqref{TrL}
%and \eqref{TrR} by  $\Tr(f) =\Tr^L(\psi_X f) =\Tr^R(f \psi_X^{-1})$. 

%%%%%%%%%%%%%%%%%%%%%%%%%%%%%%%%%%%%%%%%%%%%%%%%%%%%%%%%%
\subsection{$S$-matrix of a pre-modular category}
\label{Smat}

Let $\C$ be a   pre-modular  category with a spherical structure  $\psi$.
Let $\O(\C)$ denote the set of (isomorphism classes of) simple objects of $\C$.
For all $X,\, Y,\,Z \in \O(\C)$ let  $N_{XY}^Z$ denote the multiplicity of 
$Z$ in $X\ot Y$.

We will idetify the correspoonding twist $\theta \in \Aut(\id_\C)$ with
a collection of scalars $\theta_X\in \kk^\times,\,X\in \O(\C)$.
Let $\Tr$ and $d$ denote the trace and dimension corresponding to $\psi$.

\begin{definition}
\label{Smatr}
Let $\C$ be a pre-modular category.
The {\em $S$-matrix}  of $\C$ is  defined by 
\begin{equation}
\label{def S matrix}
S:=\left(s_{XY}\right)_{X,Y\in \O(\C)},
\quad\mbox{where } \quad
s_{XY} = \Tr(c_{Y,X}c_{X,Y}).
\end{equation}
\end{definition}

\begin{remark}
\label{Smatrem}
The $S$-matrix of $\C$ is a symmetric $n$-by-$n$ matrix where
$n=|\O(\C)|$ is the number of simple objects of $\C$.
It satisfies $s_{X^* Y^*} = s_{XY}$ for all $X,Y\in \O(\C)$.
We also  have $s_{X \be} = s_{\be X} =d_X$.
\end{remark}

\begin{definition} (\cite{T}) A pre-modular category $\C$ is said to be {\em modular} if its
$S$-matrix is non-degenerate. 
\end{definition}

\begin{example}
Let $G$ be a finite Abelian group. Let $q: G \to \kk^\times$ be a quadratic form
on $G$ and let $b: G \times G \to \kk^\times$ be the associated symmetric bilinear form.  
The $S$-matrix of the corresponding
pointed premodular category $\C(G,\,q)$ (see Section~\ref{def braid sect}) is
$\{ b(g,\,h)\}_{g,h\in G}$. Thus, $\C(G,q)$ is modular if and only if
$q$ is non-degenerate.
\end{example}

%Clearly, a pre-modular category is modular if and only  if it is non-degenerate in the 
%sense of Definition~\ref{defnondeg}.

Let  $\C$ be a pre-modular category.

\begin{proposition}
We have 
\begin{equation}
\label{sij}
s_{XY} = \theta_X^{-1}\theta_Y^{-1} \sum_{Z\in \O(\C)}
\, N_{XY}^Z \theta_Z d_Z.
\end{equation}
for all $X,Y\in \O(\C)$.
\end{proposition}
\begin{proof}
Apply $\Tr$ to both sides of formula \eqref{twist}. The right hand side 
is equal to
$\theta_X \theta_Y s_{XY}$ while the left hand side is equal to
\begin{eqnarray*}
\Tr(\theta_{X\ot Y})
&=&  \sum_{Z\in \O(\C)} \, N_{XY}^Z \Tr(\theta_Z \id_Z) \\
&=&  \sum_{Z\in \O(\C)}\, N_{XY}^Z \theta_Z d_Z,
\end{eqnarray*}
where we used additivity of $\Tr$.
\end{proof}

\begin{remark}
When $\C =\C(G,\,q)$
the relation \eqref{sij} between the twist and $S$-matrix of a premodular
category generalizes the relation \eqref{bq} between the quadrtaic form
and associated bilinear form.
\end{remark}

The elements of  $S$-matrix satisfy the following {\em Verlinde formula}
(see \cite[Theorem 3.1.12]{BK}, \cite[Lemma 2.4 (iii)]{Mu2} for a proof):
\begin{equation}
\label{ss= Nds}
s_{XY} s_{XZ} = d_X \sum_{W\in \O(\C)} \, N_{YZ}^W s_{XW},\quad X,Y,Z\in \O(\C).
\end{equation}

\begin{remark}
\label{K0 characters}
Formula \eqref{ss= Nds} can be interpreted as follows. For any fixed $X\in \O(\C)$
the map 
\begin{equation}
\label{homK0}
h_X: Y \mapsto \frac{s_{XY}}{d_X}, \quad Y\in \O(\C)
\end{equation}
gives rise to a ring homomorphism $K(\C)\to \kk$ which we will also denote $h_X$. 
That is, simple objects of $\C$  give rise to characters of the Grothendieck ring $K(\C)$.
We have $h_\be(Y) = d_Y$.
\end{remark}

The characters satisfy the following orthogonality relation:
\begin{equation}
\label{S-matr ort}
 \sum_{X\in \O(\C)}\,h_Y(X) h_Z(X^*) =0 \quad \mbox{for} \quad Y \not\cong Z.  
\end{equation}

The following result is due to Anderson, Moore, and Vafa \cite{AM, V}.

\begin{theorem} 
\label{AMV}
Let $\C$ be a premodular category.  Let $\theta$ be the twist of $\C$.
Then ${\theta_X}$ is a root of unity for all $X\in \C$.
\end{theorem}

%%%%%%%%%%%%%%%%%%%%%%%%%%%%%%%%%%%%%%%%%%%%%%%%%%%%%%%%%
\subsection{Modular categories}
\label{modular}

The categorical dimension of a fusion category $\C$ was defined in \eqref{dim def}.
When  $\C$ is premodular we have
\begin{equation}
\dim(\C)=\sum_{X\in \O(\C)}\, d_X^2,
\end{equation}
where $d$ is the dimension associated to the spherical structure of $\C$.
%Recall that  $\dim(\C)\neq 0$ by Theorem~\ref{ENO23}.

Let $E=\{ E_{XY} \}_{X,Y\in \O(\C)}$ be the square matrix such that $E_{XY}=1$
if $X=Y^*$ and $E_{XY}=0$ otherwise. 

\begin{proposition}
\label{S2=E}
Let $\C$ be a modular category and $S$ be its $S$-matrix. Then $S^2=\dim(\C)E$.
%In other words, $S^{-1} =\{ \dim(\C)^{-1} s_{XY^*} \}$.
\end{proposition}
\begin{proof}
Since $S$ is non-degenerate, 
the equality $h_Y = h_Z$ for $Y,\, Z\in \O(\C)$  holds if and only if $Y=Z$, 
where $h_Y: K(\C)\to \kk$ are the characters defined in \eqref{homK0}.

Suppose $Y\neq Z$. Using \eqref{S-matr ort} we have
$$
\sum_{X\in \O(\C)}\, s_{XY} s_{XZ^*}= \sum_{X\in \O(\C)}\, s_{XY} s_{X^*Z} =  0
$$
It remains to check that
$\sum_{X\in \O(\C)}\, s_{XY} s_{XY^*}= \dim(\C)$ for all $Y\in \O(\C)$.
We compute 
\begin{eqnarray*}
\sum_{X\in \O(\C)}\, s_{XY} s_{XY^*}
&=& \sum_{X \in \O(\C)}\, d_X  s_{XW}  \sum_{W\in \O(\C)}\,  N_{YY^*}^W \\
&=& \dim(\C) N_{YY^*}^{\be} =\dim(\C).
\end{eqnarray*}
Here the first equality is \eqref{ss= Nds}. The second equality is a
consequence of orthogonality of  characters \eqref{S-matr ort}, since
$$
\sum_{X \in \O(\C)}\, d_X  s_{XW} =  d_W \sum_{X \in \O(\C)}\, d_X  h_W(X^*)
$$
and the latter expression is equal to $\dim(\C)$ if $W=\be$ and $0$ otherwise. 
\end{proof}

\begin{corollary}
\label{verlinde}
%\textbf{(Verlinde formula).}
Let $\C$ be a modular category.
For all objects $Y,\,Z,\, W\in \O(\C)$ we have
\begin{equation}
\label{ver eqn}
\sum_{X\in \O(\C)}\, \frac{s_{XY}s_{XZ}s_{XW^*}}{d_X}
=\dim(\C) N_{YZ}^W.
\end{equation}
\end{corollary}

Thus, the $S$-matrix determines the fusion rules of $\C$.

For any $Z\in \O(\C)$ define the following
square matrices:  
\begin{equation*}
D^Z:=\left( \delta_{XY}\frac{s_{XZ}}{d_X} \right)_{X,Y\in \O(\C)}
\qquad \mbox{ and } \qquad
N^Z =  \left( N_{YZ}^W \right)_{Y, W\in \O(\C)}.
\end{equation*}

\begin{corollary}
Let $\C$ be a modular category with the $S$-matrix $S$. Then
$D^Z =S^{-1} N^Z S$ for all $Z\in \O(\C)$, i.e., conjugation 
by the $S$-matrix diagonalizes the fusion rules of $\C$.
\end{corollary}

\begin{proposition}
\label{EGdiv}
Let $\C$ be  a modular category and let $X\in \O(\C)$.
Then $\frac{\dim(\C)}{d_X^2}$ is an algebraic integer.
\end{proposition}
\begin{proof}
We compute, using Proposition~\ref{S2=E}:
\begin{equation}
\label{ratio}
\frac{\dim(\C)}{d_X^2} 
= \sum_{Y\in \O(\C)}\, \frac{s_{XY}}{d_X}  \frac{s_{XY^*}}{d_X} 
= \sum_{Y\in \O(\C)}\,  h_Y(X) h_{Y^*}(X),
\end{equation}
where $h_Y,\, Y\in \O(C),$ are characters of $K(\C)$ defined in 
\eqref{homK0}. Since  $h_Y(X)$ is
an eigenvalue of the integer matrix $N^X$, it is an algebraic integer.
Hence,  the right hand side of \eqref{ratio} is an algebraic integer.
\end{proof}

%%%%%%%%%%%%%%%%%%%%%%%%%%%%%%%%%%%%%%%%%%%%%%%%%%%%%%%%%
\subsection{Modular group representation and Galois action}
\label{sect mod group and Galois}

Modular categories have important arithmetic features that we describe next.

The {\em modular group} is, by definition, the group 
$\Gamma:= SL_2(\mathbb{Z})$ of  $2\times 2$ matrices with integer
entries and determinant $1$. 

It is known that $\Gamma$ is generated by two matrices
\begin{equation}
\mathfrak{s}:= \begin{pmatrix} 0 &-1 \\ 1& 0  \end{pmatrix}
\quad \mbox{ and }\quad
\mathfrak{t}:= \begin{pmatrix} 1 &1 \\ 0& 1  \end{pmatrix}.
\end{equation}

Let $\C$ be a modular category. 
It turns out that $\C$ gives rise to a projective representation
of $\Gamma$. This justifies the terminology.  Namely,  let $S =\left(s_{XY}\right)_{X,Y\in \O(\C)}$
be the $S$-matrix of $S$ and 
let $T=\left(t_{XY}\right)_{X,Y\in \O(\C)}$ be a diagonal matrix with entries
$t_{XY}=\delta_{X,Y}\theta_X$.

The assignments 
\begin{equation}
\label{projGamma}
\mathfrak{s} \mapsto \frac{1}{\sqrt{\dim(\C)}} S \quad \mbox{ and } \quad \mathfrak{t} \mapsto T
\end{equation}
define a projective representation $\rho: \Gamma \to GL_{|\O(\C)|}(\kk)$.
When $\C$ is the center of a fusion category this representation is linear.
%In general, it can be lifted to a linear representation $\rho: \Gamma \to GL_{|\O(\C)|}(\k)$. 

Let $N$ denote the order of $T$.  It was shown in \cite{NgS} that the kernel of $\rho$ 
is a congruence subgroup of level $N$ (i.e.,  $\Ker(\rho)$ contains the kernel  of
the natural group homomorphism $SL_2(\mathbb{Z}) \to SL_2(\mathbb{Z}/N\mathbb{Z})$). 
For Hopf algebras this result was established in \cite{SZ}.

The entries of  $S$ and $T$  are integers in $\mathbb{Q}[\xi_N]$, where $\xi_N$
is a primitive $N$th root of unity \cite{CG, dBG}.   Furthermore, matrices
in the image of $\rho$ have the following remarkable property  with respect
to the Galois group $\Gal(\mathbb{Q}(\xi_N)/\mathbb{Q})$ (see \cite{DLN}
and references therein).  For every $\sigma \in   
\Gal(\mathbb{Q}(\xi_N)/\mathbb{Q})$ the matrix $G_\sigma:= \sigma(S)S^{-1}$
is a signed permutation matrix  and 
\[
\sigma^2(\rho(M)) = G_\sigma \rho(M)  G_\sigma^{-1}
\]
for all $M \in \Gamma$.  In particular, there is a permutation $\tilde{\sigma}$ of $\O(\C)$
such that  
\[
\sigma(s_{XY})= \pm s_{X\tilde{\sigma}(Y)} \qquad \mbox{ and } \qquad 
\sigma^2(\theta_X) = \pm \theta_{\tilde{\sigma}(X)}
\]
for all $X,\,Y \in \O(\C)$.

%%%%%%%%%%%%%%%%%%%%%%%%%%%%%%%%%%%%%%%%%%%%%%%%%%%%%%%%%
\subsection{Centralizers and non-degeneracy}

Recall from \cite{Mu2} that
objects $X$ and $Y$ of a braided fusion category $\C$ are said to
{\em centralize\,}  each other if
\begin{equation} \label{monodromy-drinf}
c_{Y,X}\circ c_{X,Y} =\id_{X\ot Y}.
\end{equation}
The {\em centralizer\,} $\D'$ of a fusion subcategory $\D\subset\C$ is defined to
be the full subcategory of objects of $\C$ that centralize each object of $\D$.
It is easy to see that $\D'$ is a fusion subcategory of $\C$.
Clearly, $\D$ is symmetric if and only if $\D\subset\D'$.

\begin{definition}
\label{nondegdef}
We will say that a  braided fusion category $\C$ is {\em non-degenerate}
if $\C'=\Vec$.
\end{definition}

%A non-degenerate braided fusion category $\C\neq \Vec$ is {\em prime} if it has no 
%proper non-degenerate braided fusion subcategories other than $\Vec$.
%Clearly, a non-trivial simple  braided fusion category is prime.

For a fusion subcategory $\D$ of a non-degenerate braided fusion category $\C$ one 
has the following properties, see \cite{Mu2} and  \cite[Theorems 3.10, 3.14]{DGNO2}:
\begin{gather}
\label{double centralizer}
\D''=\D ,\\
\FPdim(\D)\FPdim(\D')=\FPdim(\C).
\end{gather}
Furthermore, if $\D$ is non-degenerate, then 
\begin{equation}
\label{Mueger's factorization}
\C \cong \D \bt \D'.
\end{equation}

\begin{example}
\label{centralizer of pt}
Let $\C$ be a non-degenerate braided fusion category.  Let $\C_{ad}$ be the adjoint
subcategory of $\C$ (see Definition~\ref{def:adjoint})
and let $\C_{pt}$ be the maximal pointed subcategory
of $\C$.  Then 
\begin{equation}
\label{Cpt and Cad}
\C_{ad}' =\C_{pt}\quad \mbox{ and } \quad \C_{pt}' =\C_{ad}.
\end{equation}
\end{example}

For the proof of the following result see  \cite{Mu1} and \cite[Proposition 3.7]{DGNO1}.

\begin{proposition}
\label{Muger1prime}
The following conditions are equivalent for a pre-modular category $\C$:
\begin{enumerate}
\item[(i)] $\C$ is modular;
\item[(ii)] $\C$ is non-degenerate, i.e., $\C'=\Vec$;
\item[(iii)] $\C$ is factorizable, i.e., the functor  $G:\C \boxtimes \C^\rev\to \Z (\C)$ defined in \eqref{G:CC-->ZC} is an equivalence.
\end{enumerate}
\end{proposition}

\begin{corollary} 
\label{centerfactorizable}
Let $\C$ be a fusion category. Then its center $\Z (\C)$ is non-degenerate.
\end{corollary}
\begin{proof} 
It is proved in \cite{ENO4} that $\Z(\C)$ is factorizable, so the result follows from Proposition~\ref{Muger1prime}.
\end{proof}

The following  Class Equation was proved in  \cite[Proposition 5.7]{ENO1}.
It is very useful  for classification of fusion categories, see Section~\ref{Sect prime power FP} below.

\begin{theorem}
\label{class eqn theorem}
Let $\A$ be a spherical fusion category.  Let $F: \Z(\A)\to \A$  be the forgetful functor.
Then 
\begin{equation}
\label{class equation}
\dim(\A) = \sum_{Z\in \O(\Z(\A))}\, [F(Z):\be] d_Z, 
\end{equation}
and $\frac{\dim(\A)}{d_Z}$ is an algebraic integer for every $Z\in \O(\Z(\A))$.
\end{theorem}
\begin{proof}
Let $I: \A \to \Z(\A)$ be the right adjoint of $F$.   We have 
\begin{equation}
\label{FI1}
FI(\be) \cong \bigoplus_{X\in \O(\A)}\, X\ot X^*
\end{equation}
and, hence,  the dimension of $I(\be)$ is equal to 
$\sum_{X\in \O(\A)}\, d_X^2 = \dim(\A)$.
On the other hand, $I(\be) \cong \oplus_{Z\in \O(\Z(\A))}\, [F(Z):\be] Z$. Taking the dimensions
of both sides of the last equation we obtain \eqref{class equation}. From Proposition~\ref{EGdiv}
we know that $\frac{\dim(\Z(\A))}{d_Z^2} = \left( \frac{\dim(\A))}{d_Z} \right)^2$ is an algebraic integer.
\end{proof}

\begin{remark}
Theorem~\ref{class eqn theorem} says that $\dim(\A)$ can be written as a sum  of algebraic integers
that are also divisors of $\dim(\A)$  in the ring of algebraic integers. 
This is an analogue of the Class Equation in group theory. 
Indeed, when $G$ is a finite group and $\A =\Rep(G)$ then simple subobjects of $I(\be)$
are in bijection with conjugacy classes of $G$ and their dimensions are cardinalities of the corresponding
conjugacy classes. 
\end{remark}

We have  the following  relation between the Frobenius-Perron  and categorical 
dimensions, \cite[Proposition 8.22]{ENO1}. 
%Recall that for any fusion category $\A$
%its Frobenius-Perron dimension  $\FPdim(\A)$ and categorical dimension $\dim(\A)$ 
%were introduced in \eqref{FPdim def} and \eqref{dim def}.

\begin{proposition}
\label{ratio dim/ FPdim}
For any spherical fusion category $\A$ over $\mathbb{C}$ the ratio $\frac{\dim(\A)}{\FPdim(\A)}$
is an algebraic integer $\leq 1$. 
\end{proposition}
\begin{proof}
Let $\C = \Z(\A)$. By  \eqref{dim < FPdim}
it suffices to show  that  $\frac{\dim(\C)}{\FPdim(\C)}$ is an algebraic integer. 
 Let $S =\{s_{XY}\}$ denote the $S$-matrix of $\C$.  The Frobenius-Perron
dimension is a homomorphism from $K(\C)$ to $\mathbb{C}$, hence, it must be of the form \eqref{homK0}.
Thus, there exists a distinguished object $X\in \C$
such that $\FPdim(Z) = \frac{s_{ZX}}{d_X}$ for all simple objects $Z$ in $\C$. Therefore,
\[
\FPdim(\C) =  \sum_Z\, \FPdim(Z)^2 = \sum_Z\,   \frac{s_{ZX}}{d_X} \frac{s_{Z^*X}}{d_X} = \frac{\dim(\C)}{d_X^2}.
\]
Thus, $\frac{\dim(\C)}{\FPdim(\C)} = d_X^2$. The latter is an algebraic integer since $d: K(\C)\to \mathbb{C}$
is a homomorphism.
\end{proof}

Now suppose that $\C$ is a pseudo-unitary non-degenerate braided fusion category over $\mathbb{C}$
(so that there is a spherical structure on $\C$ such that $d_X = \FPdim(X)$).  
One can recognize pairs of centralizing simple objects of $\C$  using the $S$-matrix.

\begin{proposition}
\label{recognizing centralizing} 
Let  $X,\, Y$ be simple objects of $\C$. Then $X$ centralizes $Y$ if and only 
if $s_{XY} = \FPdim(X) \FPdim(Y)$.
\end{proposition}
\begin{proof}
If $X$ centralizes $Y$ then $c_{Y,X}c_{X,Y} =\id_{X\ot Y}$ and
\[
s_{XY}=  d_X d_Y = \FPdim(X) \FPdim(Y).
\]
Conversely, if $s_{XY} = \FPdim(X) \FPdim(Y)$ then using formula \eqref{sij}
we obtain
\begin{eqnarray*}
\FPdim(X) \FPdim(Y) &=& |s_{XY}| = 
|  \theta_X^{-1}\theta_Y^{-1} \sum_{Z\in \O(\C)}
\, N_{XY}^Z \theta_Z d_Z |  \\
&\leq&  \sum_{Z\in \O(\C)} \, N_{XY}^Z  \FPdim(Z)  = \FPdim(X) \FPdim(Y).
\end{eqnarray*}
The above inequality must be an equality, hence $\frac{\theta_Z}{\theta_X \theta_Y} =1$ for all simple objects $Z$
contained in $X\ot Y$. By \eqref{twist} this means that $c_{Y,X}c_{X,Y} =\id_{X\ot Y}$.
\end{proof}

 %%%%%%%%%%%%%%%%%%%%%%%%%%%%%%%%%%%%%%%%%%%%%%%%%
\subsection{Equivariantization and  de-equivariantization of  braided fusion categories}
\label{Sect de-eq}

Let $\C$ be a braided fusion category and let $G$ be a group acting on  $\C$
by braided autoequivalences (i.e., each $T_g,\, g\in G,$ from \eqref{Tg} is a
braided autoequivalence of $\C$).  Then the equivariantized fusion category $\C^G$
inherits the braiding from $\C$.  Note that $\C^G$ contains a Tannakian subcategory $\Rep(G)$
such that $\Rep(G)\subset (\C^G)'$.

Here we describe the converse construction, following \cite{Br, Mu1, P}. 
Let $\D$ be a braided fusion category containing a Tannakian subcategory $\Rep(G)$ such that 
\begin{equation}
\label{G in Dprime}
\Rep(G) \subset \D'.
\end{equation}
Let $A$ be the algebra of functions
on $G$. It is a commutative algebra in $\Rep(G)$ and, hence, in $\D$.  The category $\D_G$
of $A$-modules in $\D$ has a canonical structure of a fusion category with the tensor
product $\ot_A$.  Furthermore, condition~\eqref{G in Dprime} allows to define the braiding
on  $\D_G$.  Thus, $\D_G$ is a braided fusion category, called {\em de-equivariantization} of $\D$. 
There is a canonical action of $G$ by braided autoequivalences of $\D_G$ induced by the action
of $G$ on $A$ by translations. 

As the names suggest, the above two constructions are inverses of each other. Namely,
there exist canonical braided equivalences of fusion categories:
\[
(\C^G)_G \cong \C \quad \mbox{and}\quad (\D_G)^G \cong \D. 
\]
See \cite[Section 4]{DGNO2}
for a complete treatment of equivariantization and  de-equi\-va\-ri\-an\-ti\-zation.

%%%%%%%%%%%%%%%%%%%%%%%%%%%%%%%%%%%%%%%%%%%%%%%%%
%%%%%%%%%%%%%%%%%%%%%%%%%%%%%%%%%%%%%%%%%%%%%%%%%
%%%%%%%%%%%%%%%%%%%%%%%%%%%%%%%%%%%%%%%%%%%%%%%%%%%%%%%%%
\section{Characterization of  Morita equivalence}

%%%%%%%%%%%%%%%%%%%%%%%%%%%%%%%%%%%%%%%%%%%%%%%%%
\subsection{Braided equivalences of centers}

The following  theorem was proved in \cite[Theorem 3.1]{ENO2}. It is a categorical
counterpart of the well known fact in algebra  that  Morita equivalent rings
have isomorphic centers. 

\begin{theorem}
\label{2007 criterion}
Two fusion categories $\A$ and $\B$ are categorically Morita equivalent if and only if $\Z(\A)$
and $\Z(\B)$ are equivalent as braided fusion categories
\end{theorem}
\begin{proof}
Let $\M$ be an indecomposable left $\A$-module category such that $\B=(\A^*_\M)^\op$. 
As  in  Section~\ref{sect duality}, we can view $\M$  as an $(\A \bt \B)$-module category. 
It was observed in \cite{Sch} (see also \cite{Mu-I}) that the category of 
$(\A \bt \B)$-module endofunctors of $\M$ can be identified, on the one hand,
with functors of tensor multiplication by objects of  $\Z(\A)$, and on the other hand,
with functors of tensor  multiplication by objects of $\Z(\B)$. Combined, these identifications
yield a canonical equivalence of braided fusion categories
\begin{equation}
\label{Schauenburg}
\Z(\A)\xrightarrow{\sim} \Z(\B).
\end{equation}

Conversely, let $\A$ and $\B$ be fusion categories such that there is a braided tensor equivalence
$a: \Z(\B)\xrightarrow{\sim} \Z(\A)$.
Let $F: \Z(\B) \to \B$ be the forgetful functor  and let $I: \B \to \Z(\B)$
be its right adjoint.  Then $I(\be)$ is a commutative algebra in $\Z(\B)$ and $L:=a \circ I(\be)$
 is a commutative algebra in $\Z(\A)$. We can also view $L$ as an algebra in $\A$. Let
 \[
 L =\oplus_i\,L_i
\]
be the decomposition of $L$ into the sum of indecomposable algebras in $\A$ and let $\M_i$
denote the category of right $L_i$-modules in $\A$. Then the equivalence class of $\M_i$
does not depend on $i$  and $\B \cong \A^*_{\M_i}$. See \cite[Section 3]{ENO2}
for details. 
\end{proof}

Thus, the Morita equivalence class of a fusion category $\A$ is completely determined
by its center $\Z(\A)$.

\begin{remark}
A more precise statement of Theorem~\ref{2007 criterion} is given in \cite[Theorem 1.1]{ENO3}.
Namely, the $2$-functor of taking the center gives a fully faithful  embedding 
of the $2$-category  of Morita equivalences of fusion categories into the $2$-category 
of braided fusion categories. In particular, for any fusion category $\A$ the group $\text{BrPic}(\A)$
of Morita autoequivalences of $\A$ is isomorphic to the group of braided autoequivalences
of $\Z(\A)$ and there is an equivariant bijection between the set of Morita equivalences between 
$\A$ and $\B$  and braided equivalences \eqref{Schauenburg}. See \cite[Section 5]{ENO3}
for details.
\end{remark}

%%%%%%%%%%%%%%%%%%%%%%%%%%%%%%%%%%%%%%%%%%%%%%%%%
\subsection{Recognizing  centers of extensions}

Let $G$ be a finite group and let  $\A$ be a $G$-extension of a fusion category $\B$:
\[
\A =\bigoplus_{g \in G}\, \A_g,\qquad \A_e =\B.
\]
The center of $\A$ contains a Tannakian subcategory $\E \cong \Rep(G)$ whose objects are constructed 
as follows. For every representation $\pi: G \to GL(V)$ of $G$ consider  the object $Y_\pi$ in $\Z(\A)$,
where $Y_\pi =V \ot \be$ as an object of $\A$ with the permutation isomorphism
\[
\gamma_{Y_\pi}:=\pi(g)\ot \id_X :   X \ot Y_\pi \xrightarrow{\sim} Y_\pi \ot X,\qquad X\in \A_g.
\]
Here we identified $X \ot Y_\pi$ and  $Y_\pi \ot X$ with $V \ot X$.  

The above property in fact characterizes $G$-extensions. The following statement is 
\cite[Theorem 1.3]{ENO2}.

\begin{theorem}
\label{ZC --> ext} 
Let $G$ be a finite group.  A fusion category $\A$ is  Morita equivalent to a $G$-extension
of some fusion category if and only if  $\Z(\A)$ contains a Tannakian subcategory $\E=\Rep(G)$. 
\end{theorem}

\begin{remark}
In the situation of Theorem~\ref{ZC --> ext}   consider the de-equivariantization $\E'_G$,
see Section~\ref{Sect de-eq}. Then  $\A$ is Morita  equivalent to a  $G$-graded fusion category 
$\B =\oplus_{g \in G}\, \B_g$ with  $\Z(\B_e)\cong \E'_G$. In particular,  
\begin{equation}
\label{FPdim Ae}
\FPdim(\B_e) = \frac{\FPdim(\A)}{\FPdim(\E)}.
\end{equation}
\end{remark}

\begin{corollary}
\label{GT criterion}
A fusion category $\A$ is group-theoretical if and only if $\Z(\A)$ contains a Tannakian subcategory $\E$
such that $\FPdim(\E) =\FPdim(\A)$.
\end{corollary}
\begin{proof}
From \eqref{FPdim Ae} we see that $\A$ is categorically Morita equivalent to a $G$-extension whose trivial
component is $\Vec$. Any category with the latter property is pointed and is equivalent
to $\Vec_G^\omega$ for some $3$-cocycle $\omega$.
\end{proof}

%%%%%%%%%%%%%%%%%%%%%%%%%%%%%%%%%%%%%%%%%%%%%%%%%
%%%%%%%%%%%%%%%%%%%%%%%%%%%%%%%%%%%%%%%%%%%%%%%%%
%%%%%%%%%%%%%%%%%%%%%%%%%%%%%%%%%%%%%%%%%%%%%%%%%
\section{Classification results for  fusion categories of integral dimensions}

%%%%%%%%%%%%%%%%%%%%%%%%%%%%%%%%%%%%%%%%%%%%%%%
\subsection{Fusion categories of integral dimension}

The following result is proved in \cite[Proposition 8.24]{ENO1}.

\begin{proposition}
\label{integer --> ps un}
Let $\A$ be a fusion category  such that $\FPdim(\A)$ is an integer.  Then $\A$ is pseudo-unitary.
\end{proposition}
\begin{proof}
It is shown in Proposition~\ref{ratio dim/ FPdim} that the ratio $\frac{\dim(\A)}{\FPdim(\A)}$
is an algebraic integer $\leq 1$. Let $D:= \dim(\A)$, let $D_1=D,D_2,\dots, D_N$
be algebraic conjugates of $D$, and let $g_1,\dotsm g_N$  be the elements of $Gal(\overline{\mathbb{Q}}/\mathbb{Q})$
such that $D_i = g_i(D)$. Applying Proposition~\ref{ratio dim/ FPdim} to the category $g_i(\A)$
we see that $\frac{\dim(g_i(\A))}{\FPdim(\A)}$ is an algebraic integer $\leq 1$. Therefore,
\[
\prod_{i=1}^N\,   \frac{\dim(g_i(\A))}{\FPdim(\A)}
\]
is an algebraic integer $\leq 1$. But this product is a rational number. It must be  equal to  $1$
and so all factors are equal to $1$. Thus,  ${\dim(\A)}= {\FPdim(\A)}$, as desired.
\end{proof}

\begin{corollary}
\label{canonical spherical structure in integer case}
Let $\A$ be a fusion category  such that $\FPdim(\A)$ is an integer.  Then $\A$ 
admits a unique spherical structure $a_X: X\xrightarrow{\sim} X^{**}$
with respect to which $d_X = \FPdim(X)$ for every simple object $X$.
\end{corollary}
\begin{proof}
This follows from Propositions~\ref{canonical spherical structure} and \ref{integer --> ps un}.
\end{proof}

\begin{proposition}
\label{graded dims}
Let $\A$ be a fusion category  such that $\FPdim(\A)$ is an integer. 
\begin{enumerate}
\item[(i)] For any $X\in\O(\A)$ we have $\FPdim(X) = \sqrt{n_X}$ for some
integer $n_X$. 
\item[(ii)]  The  map $\deg : \O(\A) \to \mathbb{Q}_{>0}^\times /  (\mathbb{Q}_{>0}^\times )^2$
that takes  $X\in \O(\A)$ to the image of $\FPdim(X)$ in  $\mathbb{Q}_{>0}^\times /  (\mathbb{Q}_{>0}^\times )^2$ 
is a grading of $\A$.
\end{enumerate}
\end{proposition}
\begin{proof}
Note that if $Y$ is an object  of $\A$ such  that $\FPdim(Y) \in \mathbb{Z}$ then  $\FPdim(Y_0) \in \mathbb{Z}$
for any subobject $Y_0$ of $Y$. Take $ Y =\oplus_{X\in \O(\A)}\, X\ot X^*$. We have  $\FPdim(Y) = \FPdim(\A)\in \mathbb{Z}$.
But $X\ot X^*$  is a subobject of $Y$ for every $X\in \O(\C)$. Hence, $\FPdim(X)^2 = \FPdim(X \ot X^*) \in \mathbb{Z}$.
This proves the first part. The second part is clear (cf. \cite[Theorem 3.10]{GN}).
\end{proof}

\begin{corollary}
Let $\A$ be a fusion category such that $\FPdim(\A)$ is odd. Then $\FPdim(X) \in \mathbb{Z}$ for any $X \in \O(\C)$. 
\end{corollary}
\begin{proof}
If $\A$ contains objects of irrational dimension then by Proposition~\ref{graded dims} 
it has a non-trivial grading by a $2$-group and so $\FPdim(\A)$ is even.
\end{proof}

\begin{remark}
\label{Aad dims}
Let $\A_{ad}$ be the adjoint subcategory of $\A$, see Definition~\ref{def:adjoint}.
The proof of Proposition~\ref{graded dims} shows that if $\FPdim(\A)\in \mathbb{Z}$ then
$\A_{ad}$ is integral.
\end{remark}

%%%%%%%%%%%%%%%%%%%%%%%%%%%%%%%%%%%%%%%%%%%%%%%
\subsection{Fusion categories of prime power Frobenius-Perron dimension} 
\label{Sect prime power FP}

Let $p$ be a prime number.

The following result is proved in \cite[Theorem 8.28]{ENO1}. It is a generalization
of the result of Masuoka \cite{Ma} in the theory of semisimple Hopf algebras.

\begin{proposition}
Let $\A$ be a fusion category  such that  $\FPdim(\A) =p^n$ for $n \geq 1$. Then
$\A$ has a faithful grading by $\mathbb{Z}/p\mathbb{Z}$.
\end{proposition}
\begin{proof}
By Corollary~\ref{canonical spherical structure in integer case}, $\A$ has a spherical structure such that 
$d_X = \FPdim(X)$ for all objects $X$ of $\A$. In particular, $\Z(\A)$ is a modular category.
Let $F:\Z(\A)\to \A$ and $I: \A \to \Z(\A)$ be the forgetful functor and its right adjoint.

Let $Z\in \O(\Z(\A))$ be a simple subobject of $I(\be)$.  Since $F(Z)\in \A_{ad}$ we conclude 
 by Remark~\ref{Aad dims} that  $\FPdim(Z)$ is an integer. By Proposition~\ref{EGdiv},
 $\FPdim(Z)$ is a power of $p$. Therefore, the right hand side of the Class Equation \eqref{class equation} 
 must have at least $p$ summands equal to $1$.  These summands correspond to distinct
 invertible simple subobjects of $I(\be)$.  Hence,
 the Abelian group  $G$ of such objects  is non-trivial.  Elements of  $G$
 are in bijection with tensor automorphisms of $\id_\A$.  
 Consequently,  $\A$ has a faithful grading by $\widehat{G}$. 
 Since $G$ is a $p$-group, it has a quotient isomorphic to $\mathbb{Z}/p\mathbb{Z}$ that provides
 a desired grading.
\end{proof}

\begin{corollary}
\label{pn --> nilpotent}
A fusion category of prime power Frobenius-Perron dimension is nilpotent.
\end{corollary}

\begin{corollary}
\label{Kac-Zhu}
Let $\A$ be a fusion category  such that  $\FPdim(\A) =p$. Then $\A$ is equivalent to $\Vec_{\mathbb{Z}/p\mathbb{Z}}^\omega$ 
for some $\omega \in Z^3(\mathbb{Z}/p\mathbb{Z},\, \kk^\times)$.
\end{corollary}

\begin{remark}
Corollary~\ref{Kac-Zhu} is a generalization of the classical result of Kac \cite{K} and Zhu \cite{Zhu} saying that
a Hopf algebra of prime dimension is isomorphic to the  group algebra of $\mathbb{Z}/p\mathbb{Z}$.
\end{remark} 

Recall from Definition~\ref{def:group-theoretical} 
that a fusion category $\A$ is  group-theoretical if it is categorically Morita equivalent to
a pointed fusion category. 

Let $\C$ be a braided fusion category. For any fusion subcategory $\L \subset \C$ 
let  $\L^{co}$ denote  fusion subcategory of $\C$ generated by simple objects $X\in \O(\C)$
such that  $X \ot X^* \in \L$. 

\begin{remark}
The superscript ${co}$  stands for ``commutator". The reason is that for $\C=\Rep(G)$
and $\L = \Rep(G/N)$, where $N$ is a normal subgroup of  $G$, one has $\L^{co}= \Rep(G/[G,N])$,
where $[G,N]$ is the commutator subgroup. 
\end{remark}

Theorem~\ref{main result of DGNO1} is proved in \cite{DGNO1}. Below we sketch
its proof. The reader is referred to \cite{DGNO1} for full details.

\begin{lemma}
\label{max symmetric}
Let $\C$ be a nilpotent braided fusion category.  There exists a symmetric
subcategory $\K \subset \C$ such that $(\K')_{ad} \subset \K$.
\end{lemma}
\begin{proof}
Let $\K$ be a symmetric subcategory  of $\C$. If the condition $(\K')_{ad} \subset \K$ 
is not satisfied then  fusion subcategory $\E\subset \C $ generated by $\K$ and
$\K^{co}\cap (\K^{co})'$ is symmetric and $\K \subsetneq \E$.  Thus, any maximal
symmetric subcategory of $\C$ satisfies the condition of the Lemma. 
\end{proof}

\begin{theorem}
\label{main result of DGNO1}
Let $\A$ be a fusion category such that $\FPdim(X)\in \mathbb{Z}$ for all $X \in O(\A)$
and $\FPdim(\A)= p^n$ for some prime $p$. Then $\A$ is group-theoretical.
\end{theorem} 
\begin{proof}
By Corollary~\ref{GT criterion} it suffices to show  that non-degenerate braided fusion 
category $\C := \Z(\Z(\A))\cong \Z(\A) \bt \Z(\A)^\rev$ contains a Tannakian subcategory 
$\E$ such that $\FPdim(\E) =\FPdim(\A)^2$.

This is achieved as follows. By Lemma~\ref{max symmetric},  $\Z(\A)$ 
contains a symmetric subcategory $\K$ such that $(\K')_{ad} \subset \K$. This means
that there is a faithful grading
\[
\K' = \bigoplus_{g\in G}\, \K'_g,\qquad \text{with } \quad \K'_e = \K.
\]
We view $(\K')^\rev$ as a fusion subcategory of $\Z(\A)^\rev$ and set 
\[
\E :=  \bigoplus_{g\in G}\, \K'_g \bt (\K')^\rev_g \subset \C.
\]
Then $\E$ is a symmetric subcategory such that $\E'=\E$ (so that $\FPdim(\E) =\FPdim(\A)^2$
by \eqref{double centralizer}).
In the case when $p$ is odd this subcategory $\E$ is automatically Tannakian.  When $p=2$ one can
show that existence of such $\E$ implies existence of a Tannakian subcategory of the same dimension.
\end{proof}

%%%%%%%%%%%%%%%%%%%%%%%%%%%%%%%%%%%%%%%%%%%%%%%
\subsection{Symmetric subcategories of integral braided fusion categories} 

We have seen in Theorem~\ref{ZC --> ext}  that non-trivial Tannakian subcategories of 
the center of a fusion category $\A$
are quite helpful in the study of the categorical Morita equivalence class of $\A$.  
In this Section we recall results of \cite[Section 7]{ENO2} that establish  existence
of Tannakian  subcategories of integral modular categories under certain assumptions 
on dimensions of their objects.

By Corollary~\ref{canonical spherical structure in integer case} there is  a canonical spherical
structure on $\A$ such that $d_X =\FPdim(X)$ for all objects $X$ in $\A$ and $\dim(\C)=\FPdim(\C)$.  
So we will simply talk
about {\em dimensions} of objects and fusion categories.

Let $\C$ be a non-degenerate integral braided fusion category with braiding
\[
c_{X,Y}:X \ot Y \xrightarrow{\sim} Y \ot X.
\] 

By Proposition~\ref{Muger1prime}  the category $\C$ is modular.
Let $S=\{s_{XY}\}$ denote the $S$-matrix of $\C$.
By \eqref{sij}   each entry $s_{XY}$ is a sum of $d_Xd_Y$ roots of unity, so it can be viewed as a complex
number. We have 
\begin{equation}
\label{sXY estimate}
|s_{XY}| \leq d_Xd_Y,
\end{equation}
where $|z|$ denotes the absolute value of $z\in\mathbb{C}$.

\begin{lemma}
\label{proj centr}
Let $X,\,Y$ be simple objects of $\C$.
The following conditions are equivalent:
\begin{enumerate}
\item[(i)] $|s_{XY}|  =  d_Xd_Y$,
\item[(ii)] $c_{Y,X} \circ c_{X,Y}$ equals $\id_{X\ot Y}$ times a scalar,
\item[(iii)] $X$ centralizes $Y\ot Y^*$,
\item[(iv)] $Y$ centralizes $X\ot X^*$,
\end{enumerate}
\end{lemma}
\begin{proof}
See \cite[Lemma 6.5]{GN} and \cite[Proposition 3.32]{DGNO2}.
\end{proof}

\begin{definition}
\label{def: projectively centralize}
When equivalent conditions  of Lemma~\ref{proj centr} are satisfied, we say that $X$ and $Y$ 
{\em projectively centralize} each other. 
\end{definition}

\begin{lemma}
\label{key} 
Let $X$ and $Y$ be two simple objects of $\C$ such that $d_X$ and $d_Y$ are relatively prime.
Then one of two possibilities hold:
\begin{enumerate}
\item[(i)]  $X,\,Y$ projectively centralize each other, or
\item[(ii)] $s_{XY}=0$.
\end{enumerate}
\end{lemma}
\begin{proof} 
By \eqref{homK0}, $\frac{s_{XY}}{d_X}$ and $\frac{s_{XY}}{d_Y}$ are algebraic
integers. Since $d_X$ and $d_Y$ are relatively prime, $\alpha:=\frac{s_{XY}}{d_Xd_Y}$ is also an
algebraic integer.  Indeed, if $a,\,b\in \mathbb{Z}$ are such that $ad_X+bd_Y=1$
then
\[
\frac{s_{XY}}{d_Xd_Y} =  a \frac{s_{XY}}{d_Y}  + b  \frac{s_{XY}}{d_X}. 
\]
Let $\alpha_1 = \alpha,\,\alpha_2,\, \dots, \alpha_n$ be algebraic conjugates of $\alpha$. 
Then the norm $ \alpha_1 \alpha_2 \cdots \alpha_n$ is an integer.  Since it is  $\leq 1$
in absolute value it must be either $\pm 1$ or $0$.  In the former case   $|\alpha_i|=1$
for every $i$ and so  $|\alpha|= 1$, i.e., $X,\,Y$ projectively centralize each other. 
In the latter case $\alpha =0$, i.e., $s_{XY}=0$.
\end{proof}

\begin{corollary}
\label{primpow} 
Suppose $\C$  contains a  simple object $X$ with
dimension $d_X=p^r$, where $p$ is a prime and $r>0$. Then $\C$
contains a nontrivial symmetric subcategory.
\end{corollary}
\begin{proof} 
We first show that $\E$ contains a nontrivial
proper subcategory. Let us assume that it does not. 
Take any simple $Y\ne \be$ with $d_Y$ coprime to $d_X$ (such a $Y$ must exist
since $p$ divides $\dim(\C)$ by Proposition~\ref{EGdiv}). 
We claim that $s_{XY}=0$. Indeed,
otherwise $X$ and $Y$ projectively centralize each other by Lemma~\ref{key},
so the centralizer of the category generated by $Y\ot Y^*$ is
nontrivial, and we get a nontrivial proper subcategory, a  contradiction. 

Now let us use the orthogonality of columns $(s_{XY})$ and $(d_Y)$
of the $S$-matrix: 
\[
\sum_{Y\in \O(\C)}\, \frac{s_{XY}}{d_X}\,d_Y=0.
\]
It follows that all the nonzero summands in this sum, except
the one for $Y=\be$, come from objects $Y$ of dimension divisible
by $p$. Therefore, all the 
summands in this sum except for the one for $Y=\be$ (which equals
$1$) are divisible by $p$. This is a contradiction.

Now we prove the corollary by induction in $\dim(\C)$. 
Let $\D$ be a nontrivial proper subcategory of $\C$. 
If $\D$ is degenerate, then  $\D\cap \D'$ is a nontrivial
proper symmetric subcategory of $\C$, so we are done. 
Otherwise, $\D$ is non-degenerate, and by \eqref{Mueger's factorization}
we have
$\C=\D\boxtimes \D'$. Thus $X=X_1\otimes X_2$, where $X_1\in \D$,
$X_2\in \D'$ are simple. Since the dimension of $X_1$ or $X_2$ is
a positive power of $p$, we get the desired statement from the
induction assumption applied to $\D$ or $\D'$ (which are non-degenerate braided
fusion categories of smaller dimension). 
\end{proof}

\begin{remark}
Corollary~\ref{primpow}  generalizes Burnside's theorem that a finite group $G$ 
with a conjugacy class of prime power size  can not be simple.
\end{remark}

%%%%%%%%%%%%%%%%%%%%%%%%%%%%%%%%%%%%%%%%%%%%%%%
\subsection{Solvability of fusion categories of  Frobenius-Perron dimension $p^aq^b$} 

Recall from Definition~\ref{def:weakly group-theoretical} 
that a fusion category $\A$ is solvable if it is categorically Morita equivalent to
a cyclically nilpotent fusion category. 

\begin{proposition}
A fusion category $\A$ is solvable if and only if there is a sequence of  fusion categories
\[
\A_0 =\Vec,\, \A_1,\dots, \A_n = \A
\] 
and a sequence of cyclic groups of prime order such that $\A_i$ is obtained from $\A_{i-1}$
either by a $G_i$-equivariantization or as a $G_i$-extension.
\end{proposition}
\begin{proof}
See \cite[Proposition 4.4]{ENO2}. 
\end{proof}

Let $p$ and $q$ be prime numbers.

\begin{theorem}
\label{burnside's paqb theorem}
Let $\C$ be an integral non-degenerate braided fusion category of 
Frobenius-Perron dimension $p^aq^b$.  If $\C$ is not pointed, it contains
a Tannakian subcategory $\Rep(G)$, where $G$ is a cyclic group of prime order. 
\end{theorem}
\begin{proof}
First let us show that $\C$ contains an invertible object.  Assume the contrary. 
By Proposition~\ref{EGdiv}  the dimension of  every $X\in \O(\C)$ divides $p^aq^b$.
There must be a simple object in $\C$ whose dimension is a prime power, since
otherwise the dimension of every non-identity simple object is divisible by $pq$
and  
\[
\FPdim(\C) = 1 (\mod pq), 
\]
a contradiction.  By Proposition~\ref{primpow} 
$\C$ contains a non-trivial symmetric subcategory $\E$.  This category $\E$ is not equivalent
to $\sVec$ since $\C_{pt} =\Vec$ by assumption. By Remark~\ref{non-triv tannakian} $\E$ contains
a non-trivial Tannakian subcategory $\Rep(G)$. The group $G$ is solvable
by the classical Burnside's theorem in group theory, hence, $\Rep(G)$ must contain
invertible objects.

Let $\C_{pt}$ be the maximal pointed subcategory of  $\C$.  We claim  that $\C_{pt}$ cannot
be non-degenerate. Indeed, otherwise $\C \cong \C_{pt}\bt \C_1$, where $\C_1 =\C_{pt}'$
by \eqref{Mueger's factorization}. But then the above argument $\C_1$ contains
non-trivial invertible objects which is absurd. 

Consider the symmetric fusion category $\E = \C_{pt}\cap \C_{pt}'$.
Let $n =\FPdim(\E)$. If $n>2$ then $\E$ contains a non-trivial Tannakian subcategory.
It remains to consider the case when $n =2$, i.e., when $\E =\sVec$ (this can only
happen if one of the primes $p,\,q$ is equal to $2$).  This
situation is treated in  \cite[Propositions 7.4 and 8.3]{ENO2}.
\end{proof}

\begin{corollary}
A fusion category of Frobenius-Perron dimension $p^aq^b$ is solvable.
\end{corollary}
\begin{proof}
Let $\A$ be a fusion category such that $\FPdim(\A)= p^aq^b$.  We may assume that
$\A$ is integral. Indeed,  otherwise  $\A$ is an extension of an integral fusion category 
by Remark~\ref{Aad dims}
and the result follows by induction on $\FPdim(\A)$ using Lemma~\ref{extension Morita}.

The category $\Z(\A)$ satisfies the hypothesis of Theorem~\ref{burnside's paqb theorem}  
and, hence, contains a Tannakian subcategory $\Rep(G)$, where $G$ is a cyclic group of prime order.
By Theorem~\ref{ZC --> ext}  $\A$ is categorically Morita equivalent to a $G$-extension
of some fusion category $\B$. Since $\FPdim(\B)\leq \FPdim(\A)$ the result follows by induction.
\end{proof}

%%%%%%%%%%%%%%%%%%%%%%%%%%%%%%%%%%%%%%%%%%%%%%%
\subsection{Other results and open problem} 

Using orthogonality of columns of the $S$-matrix of a modular category one can prove several other classification results.
In particular, it was shown in \cite{ENO2} that integral fusion categories of dimension $pqr$, where $p,\,q,\, r$
are primes are group-theoretical.  It was also shown there that fusion categories of dimension $60$
are weakly group-theoretical.  

In a different direction, it was shown in \cite{NatP} that integral braided fusion category $\A$ 
such that every simple object of $\A$ has Frobenius-Perron dimension at most 2 is solvable.

The following natural question was asked in \cite{ENO2}. The answer is unknown to the author.

\begin{question}
Is every integral fusion category weakly group-theoretical?
\end{question}

Perhaps the arithmetic properties of $S$-matrices discussed in Section~\ref{sect mod group and Galois}
can be used in order to answer this question.

%%%%%%%%%%%%%%%%%%%%%%%%%%%%%%%%%%%%%%%%%%%%%%%%%%%%%%%%%%%%%
%%%%%%%%%%%%%%%%%%%%%%%%%%%%%%%%%%%%%%%%%%%%%%%%%%%%%%%%%%%%%
%%%%%%%%%%%%%%%%%%%%%%%%%%%%%%%%%%%%%%%%%%%%%%%%%%%%%%%%%%%%%
\bibliographystyle{plain}

\end{document}